\documentclass{article}
\usepackage{etex}
\usepackage{amsmath}
\usepackage{amsthm}
\usepackage{amssymb}
\usepackage{amsfonts}
\usepackage{fancyhdr}
\usepackage{pgf,tikz}

\usetikzlibrary{arrows}
\usetikzlibrary{patterns}
\usepackage{wrapfig}
\usepackage{caption}

\theoremstyle{plain}
\newtheorem{theorem}{Theorem}[section]
\newtheorem{lemma}[theorem]{Lemma}

\newtheorem{conjecture}[theorem]{Conjecture}
\newtheorem{question}[theorem]{Question}

\newcommand{\R}{\mathbb{R}}
\newcommand{\Q}{\mathbb{Q}}
\newcommand{\N}{\mathbb{N}}
  
\theoremstyle{definition}\newtheorem{definition}[theorem]{Definition}

\theoremstyle{remark}
\theoremstyle{remark}

\begin{document}
\title{Unions of regular polygons with large perimeter-to-area ratio}

\author{Viktor Kiss\thanks{Partially supported by the
Hungarian Scientific Foundation grants no.~104178, 105645.}
\ and
Zolt\'an Vidny\'anszky\thanks{Partially supported by the
Hungarian Scientific Foundation grant no.~104178.}}

\maketitle

\begin{abstract}
  T.~Keleti \cite{keleti} asked, whether the ratio of the perimeter and the 
  area of a finite union of unit squares is always at most 4. In this paper 
  we present an example where the ratio is greater than 4. We also answer the analogous question for regular triangles negatively and list a number of open problems. 
\end{abstract}

\section{Introduction}
Tam\'as Keleti \cite{keleti} proved that if we take a finite union of unit squares, 
then the perimeter-to-area ratio of the union cannot be arbitrarily large. In fact, he proved a general result: the perimeter to the area ratio of a finite union of congruent compact convex sets is bounded.

One would be tempted to think that the upper bound is realised by a single set, however this is not the case, even if we consider solely compact convex polygons. Gyenes \cite{gyenesszak} gave an example, where the perimeter-to-area ratio of the union of two polygons
exceeds the perimeter-to-area ratio of a single one. On the other hand for discs this statement holds true. 

So it is very natural to ask the following:

\begin{question}
 (Keleti) Is it true that the perimeter-to-area ratio of a single regular $n$-gon with side length $1$ maximises the perimeter-to-area ratio of the union of regular $n$-gons with side length $1$?
\end{question}

Gyenes gave a new proof for the boundedness of the ratio, improving the 
upper bound to approximately $5.6$ in the case of squares. Also, he proved that the upper bound is $4$, if we consider
squares with common centre or with sides parallel to the axis, or if we consider the union of two squares (see \cite{gyenes},\cite{gyenesszak}).

P. Humke, C. Marcott, B. Mellem and C. Stiegler \cite{humke},\cite{humke1} investigated the differentiation properties of the perimeter and area functions related to Keleti's question. This question has also generated quite some interest on MathOverflow
 \cite{dome}, with comments e.g. from W. T. Gowers. 

In this paper we give a negative answer to this question for $n=3$ and $n=4$ with completely elementary proofs. We present some examples with large perimeter-to-area ratio for $n=4$ and finally we list a number of open problems. The idea of the counterexamples is originated from the results obtained using a  probabilistic computer algorithm. In our experiments, we used the open source JTS Topology Suite library.

\textbf{Acknowledgement.} We are very grateful to Tam\'as Keleti, M\'arton Elekes and Andr\'as M\'ath\'e for their useful remarks and suggestions.

\section{Results}
We will denote the area of a polygon $P$ by $a(P)$, and the perimeter by $p(P)$. $O$ stands for the origin. First we need some technical definitions.

\begin{definition}
 Let $k$ and $n$ be coprime natural numbers. A \textit{basic $(k,n)$-setup} is $k$ many regular $n$-gons with side length $1$ so that all of the $n$-gons have origin centre and their vertices form a regular $kn$-gon.
\end{definition}
\begin{definition}
Suppose that we have a finite collection of regular polygons $P_1,\dots,P_k$ in general position, i.e., no vertex lies on the side of another polygon. Suppose that the boundary of $\bigcup_{1 \leq j \leq k} P_i$ is a closed simple (not self-intersecting) polygonal chain. Let $A_1,\dots,A_l$ be an enumeration of the vertices of the boundary. We call the translation by $v \in \R^2$ of one of the polygons $P_i$ \textit{pattern preserving}, if for every $t\in [0,1]$ 
\begin{itemize}
                                                                                                                                                                                                                                                                                                                                                                                                                                                                                    \item the polygons $P_1, \dots, P_i+tv,\dots,P_k$ are in a general position
	    \item the boundary of $(\bigcup_{j \not = i} P_j) \cup (P_i+tv)$ is also a closed 
	    simple polygonal chain with $l$ vertices, and its vertices can be enumerated as $A'_1,\dots,A'_l$ so that for all $1 \leq j \leq k$ and $1\leq  m \leq l$ we have \[A'_m \in P_j \iff A_m \in P_j \text{ if } j \not =i\] and \[A'_m \in P_i+tv \iff A_m \in P_i. \]																										  
                                                                                                                                                                                                                                                                                                                                                                                                                                                                                   \end{itemize}
  The translation is called \textit{regular}, if $v$ is parallel to the vector from the centre of $P_i$ to one of the vertices of $P_i$.
\end{definition}
The latter definition intuitively means that the translation preserves the pattern of intersections on the boundary of the union of the polygons. 

\begin{definition}
Consider a basic $(k,n)$-setup.  Let us mark 
  one vertex of each $n$-gon in such a way that these marked vertices 
  form a regular $k$-gon with centre $O$ and denote them by 
  $A_1,\dots,A_k$. It is easy to see that for a small enough $\varepsilon>0$ the translations by 
  $\varepsilon \cdot \overrightarrow{OA}_1,\dots,\varepsilon \cdot \overrightarrow{OA}_k$, realised in this order, are pattern preserving regular transformations. After applying these transformations to the polygons, the collection of the translated $n$-gons is called an \textit{shifted $(k,n)$-setup}.

\end{definition}

An instance of a shifted $(5,4)$-setup can be seen on Figure \ref{negyz}.
\subsection{A counterexample of 5 squares}
\begin{theorem}
 \label{t:squares}
 There exists a polygon that is the union of $5$ squares with side length $1$ for which the perimeter-to-area ratio is greater then $4$. 
\end{theorem}

First we reiterate the proof of Gyenes for a basic $(k,n)$-setup.

\begin{lemma}
  \label{common_centre}
  The perimeter-to-area ratio of the union of the $n$-gons in a basic $(k,n)$-setup is equal to the perimeter-to-area ratio of a single $n$-gon.
\end{lemma}

\begin{proof}
  The union is a polygon, let us denoted the vertices by $A_0A_1\dots A_l$ and the distance of the line segment $AB$ from the origin by $d_{AB}$. Now clearly $a(A_0A_1\dots A_l)=\frac{1}{2}(|A_0A_1|d_{A_0A_1}+|A_1A_2|d_{A_1A_2}+\dots+|A_lA_0|d_{A_lA_0})$. But since $A_iA_{i+1}$ are segments which are subsets of the boundaries of congruent regular $n$-gons with origin centre, we have that $d_{A_0A_1}=d_{A_1A_2}=\dots=d_{A_lA_0}=d$. So \[\frac{p(A_0A_1\dots A_l)}{a(A_0A_1\dots A_l)}=\frac{2}{d}\] which is the same as in the case of a single $n$-gon. 
\end{proof}

\textit{Proof of Theorem \ref{t:squares}.}
  We begin with a basic $(5,4)$-setup. Let us denote the squares in 
  the construction by $P_1$, $P_2$, $P_3$, $P_4$ and $P_5$. We move the polygons to a shifted $(5,4)$-setup by translating to the directions of 
  $A_1$, $B_1$, $C_1$, $D_1$ and $E_1$ respectively, as in Figure \ref{negyz}. We will denote 
  the other three vertices of $P_1$ by $A_2$, $A_3$ and $A_4$, and 
  suppose that they lie on the boundary in that order. We will denote 
  the vertices of the other polygons similarly. 
  Also, we denote the vertices of the translate of $P_1$ by
  $A_1'$, $A_2'$, $A_3'$ and $A_4'$, and use analogous notations for 
  the translates of the other polygons. 
  
  From the previous lemma we have that the perimeter-to-area ratio of 
  the union of the polygons in the basic $(5,4)$-setup is exactly 4. We show that the perimeter of the union of the polygons in the shifted setup will remain the same while 
  the area decreases.
\begin{lemma}
  The perimeter of the union of the squares in the shifted and in the basic setup are equal.
\end{lemma}
\begin{wrapfigure}{r}{4cm}
\definecolor{uuuuuu}{rgb}{0.27,0.27,0.27}
\definecolor{xdxdff}{rgb}{0.49,0.49,1}
\definecolor{qqqqff}{rgb}{0,0,1}
\begin{tikzpicture}[line cap=round,line join=round,>=triangle 45,x=1.6cm,y=1.6cm]
\clip(-1.52,-0.1) rectangle (0.55,1.14);
\draw (0,0.6)-- (-0.6,0);
\draw (-1.36,0.39)-- (-0.6,0);
\draw (-1.44,0.22)-- (-1.36,0.39);
\draw (-0.86,0.14)-- (0,1);
\draw (-0.6,0.4)-- (-0.6,0);
\draw (0.33,0.27)-- (0,0.6);
\draw (0,1)-- (0.5,0.5);
\draw (0,0.6)-- (0,1);
\begin{scriptsize}
\fill [color=black] (0,1) circle (1.5pt);
\draw[color=black] (0.09,1.08) node {$A'_1$};
\fill [color=black] (0,0.6) circle (1.5pt);
\draw[color=black] (0.095,0.69) node {$A_1$};
\fill [color=black] (-0.6,0) circle (1.5pt);
\draw[color=black] (-0.5,-0.03) node {$R$};
\fill [color=black] (-1.36,0.39) circle (1.5pt);
\draw[color=black] (-1.29,0.47) node {$E_2$};
\fill [color=black] (-0.86,0.14) circle (1.5pt);
\draw[color=black] (-0.9,0.05) node {$M$};
\fill [color=black] (-0.6,0.4) circle (1.5pt);
\draw[color=black] (-0.68,0.46) node {$N$};
\end{scriptsize}
\end{tikzpicture}
\captionof{figure}{}
   \label{perim}
\end{wrapfigure}
\textit{Proof.}
  Let $\varepsilon$ be the size of the shift. 
  First, shift one square only, say $P_1$. Let $T$ be a triangle with 
  sides $a, b$ and $c$, where $c = \varepsilon$ and the angle opposite 
  of side $a$ is $45^\circ$, while the angle opposite of side $b$ is 
  $63^\circ$. 
  We use the notations of Figure \ref{perim}. As can be seen, the 
  triangle $RMN$ is congruent to $T$, since $\angle MNR = 45^\circ$ 
  and $\angle NRM = 108^\circ - 45^\circ = 63^\circ$, 
  where $R$ was the intersection 
  of a segment lying on $A_1$ and the adjacent segment lying on $E_2$, 
  $M$ is the intersection of the translated square with the same 
  segment, and $N$ is on the line $MA'_1$ such that $A_1A'_1$ is 
  parallel to $RN$. Hence the length of the segment passing through 
  $A_1$ increases with $MN = b$, the adjacent segment, passing through 
  $E_2$ decreases with $MR = a$. 
  
  It is easy to see that it is the same with the other segment 
  lying on $A_1$. Meanwhile exactly the opposite happens near the 
  vertex $A_3$, since the figure is similar and the 
  length of the segments lying on $A_3$ decreases with $b$ and the 
  length of the adjacent segments increases with $b$. Hence the 
  perimeter of the union 
  remains the same considering the changes only near this two 
  vertices.

  \begin{wrapfigure}{l}{2cm}
\definecolor{uuuuuu}{rgb}{0.26666666666666666,0.26666666666666666,0.26666666666666666}
\definecolor{qqqqff}{rgb}{0.0,0.0,1.0}
\begin{tikzpicture}[line cap=round,line join=round,>=triangle 45,x=1.5cm,y=1.5cm]
\clip(-1.1,-0.8) rectangle (-0.2,1.3);
\draw (-0.4,-0.2)-- (-1.0,0.4);
\draw (-1.0,0.0)-- (-0.4,0.6);
\draw (-0.4,-0.2)-- (-0.6487555697199214,-0.6882102944352365);
\draw (-0.5350160607534188,-0.4649839392465812)-- (-1.0,0.0);
\draw (-0.6476013740677972,1.0859450579229692)-- (-0.4,0.6);
\draw (-0.5350160607534187,0.8649839392465811)-- (-1.0,0.4);
\draw (-0.6476013740677972,1.0859450579229692)-- (-0.4763614085399079,1.1731961783299776);
\draw (-0.6487555697199214,-0.6882102944352365)-- (-0.4825512363503266,-0.7728956321033013);
\draw (-0.5999999999999999,0.8)-- (-0.4,0.6);
\draw (-0.4,-0.2)-- (-0.6000000000000001,-0.39999999999999997);
\begin{scriptsize}
\draw [fill=black] (-1.0,0.0) circle (1.5pt);
\draw[color=black] (-0.85,-0.01) node {$A_2$};
\draw [fill=black] (-1.0,0.4) circle (1.5pt);
\draw[color=black] (-0.85,0.4) node {$A'_2$};
\draw [fill=black] (-0.4,-0.2) circle (1.5pt);
\draw[color=black] (-0.36173535045143507,-0.1) node {$F$};
\draw [fill=black] (-0.4,0.6) circle (1.5pt);
\draw[color=black] (-0.34,0.7) node {$X$};
\draw [fill=black] (-0.6487555697199214,-0.6882102944352365) circle (1.5pt);
\draw[color=black] (-0.78,-0.61) node {$E_3$};
\draw [fill=black] (-0.5350160607534188,-0.4649839392465812) circle (1.5pt);
\draw[color=black] (-0.42,-0.44) node {$G$};
\draw [fill=black] (-0.6476013740677972,1.0859450579229692) circle (1.5pt);
\draw[color=black] (-0.6014174294324705,1.23) node {$B_1$};
\draw [fill=black] (-0.5350160607534187,0.8649839392465811) circle (1.5pt);
\draw[color=black] (-0.43,0.9373325851482243) node {$Z$};
\draw [fill=black] (-0.5999999999999999,0.8) circle (1.5pt);
\draw[color=black] (-0.64,0.8748068254140412) node {$Y$};
\draw [fill=black] (-0.6000000000000001,-0.39999999999999997) circle (1.5pt);
\draw[color=black] (-0.76,-0.42) node {$H$};
\end{scriptsize}
\end{tikzpicture}
\captionof{figure}{}
   \label{perim2}
\end{wrapfigure}
  
  Now we take a look at the segments lying on $A_2$. 
  Let $T'$ be a triangle with sides $a', b'$ and $c'$, where 
  $c' = \varepsilon / \sqrt{2}$, and the angle opposite 
  of side $a'$ is $18^\circ$, while the angle opposite of side $b$ is 
  $90^\circ$. 
  We use the notations of Figure \ref{perim2}, where the node $X$ is 
  the intersection of two adjacent segment of the boundary of the 
  original squares passing through $B_1$ and $B_2$, $Z$ is the 
  intersection of the translated segment through $A'_2$ and the one 
  through $B_1$, while $Y$ is a point on $A'_2Z$ such that $XY$ is 
  perpendicular to $ZA'_2$. It is easy to see that triangle $XYZ$ is 
  congruent to $T'$, since $\angle XYZ = 90^\circ$ and 
  $\angle ZXY = 108^\circ - 90^\circ = 18^\circ$. 
  And the triangle $FGH$ is also congruent to $T'$, where the points 
  $F, G$ and $H$ are achieved similarly to $X, Y$ and $Z$, only they 
  are on the other side of the point $A_2$ and $A'_2$. 
  
  The upper segment in the figure, passing through $A_2$ decreased with 
  $\varepsilon / \sqrt{2}$ and increased with $YZ = a'$. The lower 
  one increased with $\varepsilon / \sqrt{2}$ and decreased with 
  $HG = a'$. The length of the segment on the boundary adjacent to the 
  upper segment passing through $A_2$ decreased with $XZ = b'$, and 
  the length of the segment adjacent to the lower segment passing 
  through $A_2$ increased with $FG = b'$, hence the changes near 
  vertex $A_2$ cancel each other out, the boundary remains the same. A similar argument can be said about the line segments around $A_4$, so 
  the perimeter does not change while translating $P_1$. 
  
  Notice that the proof here only used that the angles of the adjacent segments 
  is the same as in a basic $(5,4)$-setup, and that the pattern 
  on the boundary is preserved, hence shifting the squares one by one, 
  the same argument yields that the perimeter remains the same.
\qed

\begin{lemma}
  The area of the union of the squares in a shifted $(5, 4)$-setup is 
  less than the area in a basic $(5, 4)$-setup.
\end{lemma}

\begin{wrapfigure}{r}{2cm}
	\centering
	\definecolor{zzttqq}{rgb}{0.6,0.2,0}
	\definecolor{uuuuuu}{rgb}{0.267,0.267,0.267}
	\begin{tikzpicture}[line cap=round,line join=round,>=triangle 45,x=12.0cm,y=12.0cm]
	\clip(0.33,0.457) rectangle (0.454,0.687);
	\fill[color=zzttqq,fill=zzttqq,fill opacity=0.1] (0.395,0.679) -- (0.414,0.669) -- (0.359,0.561) -- (0.445,0.475) -- (0.43,0.46) -- (0.333,0.557) -- cycle;
	\draw [color=zzttqq] (0.395,0.679)-- (0.414,0.669);
	\draw [color=zzttqq] (0.414,0.669)-- (0.359,0.561);
	\draw [color=zzttqq] (0.359,0.561)-- (0.445,0.475);
	\draw [color=zzttqq] (0.445,0.475)-- (0.43,0.46);
	\draw [color=zzttqq] (0.43,0.46)-- (0.333,0.557);
	\draw [color=zzttqq] (0.333,0.557)-- (0.395,0.679);
	\begin{scriptsize}
	\fill [color=black] (0.43,0.46) circle (1.5pt);
	\draw[color=black] (0.409,0.466) node {$B$};
	\fill [color=black] (0.333,0.557) circle (1.5pt);
	\draw[color=black] (0.337,0.54) node {$C$};
	\fill [color=black] (0.359,0.561) circle (1.5pt);
	\draw[color=black] (0.374,0.562) node {$F$};
	\fill [color=black] (0.445,0.475) circle (1.5pt);
	\draw[color=black] (0.445,0.49) node {$A$};
	\fill [color=black] (0.395,0.679) circle (1.5pt);
	\draw[color=black] (0.380,0.675) node {$D$};
	\fill [color=black] (0.414,0.669) circle (1.5pt);
	\draw[color=black] (0.418,0.655) node {$E$};
	\end{scriptsize}
	\end{tikzpicture}
\captionof{figure}{}
   \label{old}
\end{wrapfigure}

\textit{Proof.} Let $\varepsilon$ be the size of the shift and let $x$ be 
  the common length of the segments on the boundary of the original construction, 
  with centres at the origin. Let $T$ be the area of the hexagon $ABCDEF$, 
  where $EF = FA = x$, $AB = DE = \varepsilon / \sqrt{2}$, the angles at 
  $A$, $B$, $D$ and $E$ are right angles, and the angle at $C$ is 
  $108^\circ$ (Figure \ref{old}). Let $t$ be the area of the same hexagon, only this time the 
  length of the sides $BC$ and $CD$ will equal $x$. It is easy to see 
  that $t < T$. 
  
  As it is shown on Figure \ref{negyz}, let us use 
  the notations $P = A_1' A_2' \cap E_1'E_2'$, 
  $Q = E_1 E_2 \cap E_2'E_3'$, $R = A_1 A_2 \cap E_1E_2$, 
  $T = C_3 C_4 \cap B_3B_4$, $U = B_1 B_4 \cap B_3'B_4'$, 
  $V = C_3' C_4' \cap B_3'B_4'$. Because of the symmetry of the construction, the area added to the union 
  is exactly ten times the area of the polygon $A_1 A'_1 P E'_2 Q R$ which 
  is $t - \frac{\varepsilon^2}{4}$, since $A_1R = RE_2 = x$. 
  The area subtracted from the union is ten times the area of the 
  hexagon $C_3 C'_3 V U B_4 T$ which is $T - \frac{\varepsilon^2}{4}$, 
  because $C_3 T = TB_4 = x$. So the area subtracted is greater than the 
  area added, which finishes the proof. \qed

Hence the proof of the theorem is also complete.
\qed
 \begin{center} 
\definecolor{uuuuuu}{rgb}{0.26667,0.26667,0.26667}
\definecolor{zzttqq}{rgb}{0.6,0.2,0}
\definecolor{qqqqff}{rgb}{0,0,1}
\definecolor{ttffqq}{rgb}{0.2,1,0}
\definecolor{ffqqqq}{rgb}{1,0,0}
\begin{tikzpicture}[line cap=round,line join=round,>=triangle 45,x=6cm,y=6cm]
	\draw[->,color=black] (-0.74821,0) -- (0.83058,0);
	\foreach \x in {-0.6,-0.4,-0.2,0.2,0.4,0.6,0.8}
	\draw[shift={(\x,0)},color=black] (0pt,2pt) -- (0pt,-2pt) node[below] {\footnotesize $\x$};
	\draw[->,color=black] (0,-0.76199) -- (0,0.81515);
	\foreach \y in {-0.6,-0.4,-0.2,0.2,0.4,0.6,0.8}
	\draw[shift={(0,\y)},color=black] (2pt,0pt) -- (-2pt,0pt) node[left] {\footnotesize $\y$};
	\draw[color=black] (0pt,-10pt) node[right] {\footnotesize $0$};
	\clip(-0.74821,-0.76199) rectangle (0.83058,0.81515);

	\fill[color=ffqqqq,fill=ffqqqq,fill opacity=1.0] (0.70711,0) -- (0.73711,0) -- (0.63632,0.10078) -- (0.61043,0.09668) -- cycle;
\fill[color=ffqqqq,fill=ffqqqq,fill opacity=1.0] (0.63632,0.10078) -- (0.61043,0.09668) -- (0.66249,0.19887) -- (0.68177,0.18998) -- cycle;
	\fill[color=ttffqq,fill=ttffqq,fill opacity=1.0] (0.66361,-0.20106) -- (0.6725,-0.21851) -- (0.55067,-0.28058) -- (0.52731,-0.26868) -- cycle;
\fill[color=ttffqq,fill=ttffqq,fill opacity=1.0] (0.55067,-0.28058) -- (0.52731,-0.26868) -- (0.54779,-0.39799) -- (0.57206,-0.41563) -- cycle;

	\fill[color=zzttqq,fill=zzttqq,fill opacity=0.1] (0.70711,0) -- (0,0.70711) -- (-0.70711,0) -- (0,-0.70711) -- cycle;
	\fill[color=zzttqq,fill=zzttqq,fill opacity=0.1] (0.21851,0.6725) -- (-0.6725,0.21851) -- (-0.21851,-0.6725) -- (0.6725,-0.21851) -- cycle;
	\fill[color=zzttqq,fill=zzttqq,fill opacity=0.1] (-0.57206,0.41563) -- (-0.41563,-0.57206) -- (0.57206,-0.41563) -- (0.41563,0.57206) -- cycle;
	\fill[color=zzttqq,fill=zzttqq,fill opacity=0.1] (-0.57206,-0.41563) -- (0.41563,-0.57206) -- (0.57206,0.41563) -- (-0.41563,0.57206) -- cycle;
	\fill[color=zzttqq,fill=zzttqq,fill opacity=0.1] (0.21851,-0.6725) -- (0.6725,0.21851) -- (-0.21851,0.6725) -- (-0.6725,-0.21851) -- cycle;
	\fill[color=zzttqq,fill=zzttqq,fill opacity=0.1] (0.28058,-0.55067) -- (0.21851,-0.6725) -- (0.09668,-0.61043) -- (0,-0.70711) -- (-0.09668,-0.61042) -- (-0.21851,-0.6725) -- (-0.28058,-0.55067) -- (-0.41563,-0.57206) -- (-0.43702,-0.43702) -- (-0.57206,-0.41563) -- (-0.55067,-0.28058) -- (-0.6725,-0.21851) -- (-0.61042,-0.09668) -- (-0.70711,0) -- (-0.61042,0.09668) -- (-0.6725,0.21851) -- (-0.55067,0.28058) -- (-0.57206,0.41563) -- (-0.43702,0.43702) -- (-0.41563,0.57206) -- (-0.28058,0.55067) -- (-0.21851,0.6725) -- (-0.09668,0.61042) -- (0,0.70711) -- (0.09668,0.61043) -- (0.21851,0.6725) -- (0.28058,0.55067) -- (0.41563,0.57206) -- (0.43702,0.43702) -- (0.57206,0.41563) -- (0.55067,0.28058) -- (0.6725,0.21851) -- (0.61043,0.09668) -- (0.70711,0) -- (0.61043,-0.09668) -- (0.6725,-0.21851) -- (0.55067,-0.28058) -- (0.57206,-0.41563) -- (0.43702,-0.43702) -- (0.41563,-0.57206) -- (0.28058,-0.55067) -- cycle;
	\fill[color=zzttqq,fill=zzttqq,fill opacity=0.1] (0.73711,0) -- (0.03,0.70711) -- (-0.67711,0) -- (0.03,-0.70711) -- cycle;
	\fill[color=zzttqq,fill=zzttqq,fill opacity=0.1] (0.22778,0.70103) -- (-0.66323,0.24704) -- (-0.20924,-0.64397) -- (0.68177,-0.18998) -- cycle;
	\fill[color=zzttqq,fill=zzttqq,fill opacity=0.1] (-0.59633,0.43326) -- (-0.4399,-0.55443) -- (0.54779,-0.39799) -- (0.39136,0.58969) -- cycle;
	\fill[color=zzttqq,fill=zzttqq,fill opacity=0.1] (-0.59633,-0.43326) -- (0.39136,-0.58969) -- (0.54779,0.39799) -- (-0.4399,0.55443) -- cycle;
	\fill[color=zzttqq,fill=zzttqq,fill opacity=0.1] (0.22778,-0.70103) -- (0.68177,0.18998) -- (-0.20924,0.64397) -- (-0.66323,-0.24704) -- cycle;
	\fill[color=zzttqq,fill=zzttqq,fill opacity=0.1] (0.29249,-0.57404) -- (0.22778,-0.70103) -- (0.10078,-0.63632) -- (0.03,-0.70711) -- (-0.09258,-0.58453) -- (-0.20924,-0.64397) -- (-0.26868,-0.52731) -- (-0.4399,-0.55443) -- (-0.45556,-0.45556) -- (-0.59633,-0.43326) -- (-0.57404,-0.29249) -- (-0.66323,-0.24704) -- (-0.58453,-0.09258) -- (-0.67711,0) -- (-0.58453,0.09258) -- (-0.66323,0.24704) -- (-0.57404,0.29249) -- (-0.59633,0.43326) -- (-0.45556,0.45556) -- (-0.4399,0.55443) -- (-0.26868,0.52731) -- (-0.20924,0.64397) -- (-0.09258,0.58453) -- (0.03,0.70711) -- (0.10078,0.63632) -- (0.22778,0.70103) -- (0.29249,0.57404) -- (0.39136,0.58969) -- (0.41848,0.41848) -- (0.54779,0.39799) -- (0.52731,0.26868) -- (0.68177,0.18998) -- (0.63632,0.10078) -- (0.73711,0) -- (0.63632,-0.10078) -- (0.68177,-0.18998) -- (0.52731,-0.26868) -- (0.54779,-0.39799) -- (0.41848,-0.41848) -- (0.39136,-0.58969) -- (0.29249,-0.57404) -- cycle;
	\fill[color=zzttqq,fill=zzttqq,fill opacity=0.1] (0.29249,-0.57404) -- (0.22778,-0.70103) -- (0.10078,-0.63632) -- (0.03,-0.70711) -- (-0.09258,-0.58453) -- (-0.20924,-0.64397) -- (-0.26868,-0.52731) -- (-0.4399,-0.55443) -- (-0.45556,-0.45556) -- (-0.59633,-0.43326) -- (-0.57404,-0.29249) -- (-0.66323,-0.24704) -- (-0.58453,-0.09258) -- (-0.67711,0) -- (-0.58453,0.09258) -- (-0.66323,0.24704) -- (-0.57404,0.29249) -- (-0.59633,0.43326) -- (-0.45556,0.45556) -- (-0.4399,0.55443) -- (-0.26868,0.52731) -- (-0.20924,0.64397) -- (-0.09258,0.58453) -- (0.03,0.70711) -- (0.10078,0.63632) -- (0.22778,0.70103) -- (0.29249,0.57404) -- (0.39136,0.58969) -- (0.41848,0.41848) -- (0.54779,0.39799) -- (0.52731,0.26868) -- (0.68177,0.18998) -- (0.63632,0.10078) -- (0.73711,0) -- (0.63632,-0.10078) -- (0.68177,-0.18998) -- (0.52731,-0.26868) -- (0.54779,-0.39799) -- (0.41848,-0.41848) -- (0.39136,-0.58969) -- (0.29249,-0.57404) -- cycle;
	\fill[color=zzttqq,fill=zzttqq,fill opacity=0.1] (0.29249,-0.57404) -- (0.22778,-0.70103) -- (0.10078,-0.63632) -- (0.03,-0.70711) -- (-0.09258,-0.58453) -- (-0.20924,-0.64397) -- (-0.26868,-0.52731) -- (-0.4399,-0.55443) -- (-0.45556,-0.45556) -- (-0.59633,-0.43326) -- (-0.57404,-0.29249) -- (-0.66323,-0.24704) -- (-0.58453,-0.09258) -- (-0.67711,0) -- (-0.58453,0.09258) -- (-0.66323,0.24704) -- (-0.57404,0.29249) -- (-0.59633,0.43326) -- (-0.45556,0.45556) -- (-0.4399,0.55443) -- (-0.26868,0.52731) -- (-0.20924,0.64397) -- (-0.09258,0.58453) -- (0.03,0.70711) -- (0.10078,0.63632) -- (0.22778,0.70103) -- (0.29249,0.57404) -- (0.39136,0.58969) -- (0.41848,0.41848) -- (0.54779,0.39799) -- (0.52731,0.26868) -- (0.68177,0.18998) -- (0.63632,0.10078) -- (0.73711,0) -- (0.63632,-0.10078) -- (0.68177,-0.18998) -- (0.52731,-0.26868) -- (0.54779,-0.39799) -- (0.41848,-0.41848) -- (0.39136,-0.58969) -- (0.29249,-0.57404) -- cycle;
	\draw [color=zzttqq] (0.70711,0)-- (0,0.70711);
	\draw [color=zzttqq] (0,0.70711)-- (-0.70711,0);
	\draw [color=zzttqq] (-0.70711,0)-- (0,-0.70711);
	\draw [color=zzttqq] (0,-0.70711)-- (0.70711,0);
	\draw [color=zzttqq] (0.21851,0.6725)-- (-0.6725,0.21851);
	\draw [color=zzttqq] (-0.6725,0.21851)-- (-0.21851,-0.6725);
	\draw [color=zzttqq] (-0.21851,-0.6725)-- (0.6725,-0.21851);
	\draw [color=zzttqq] (0.6725,-0.21851)-- (0.21851,0.6725);
	\draw [color=zzttqq] (-0.57206,0.41563)-- (-0.41563,-0.57206);
	\draw [color=zzttqq] (-0.41563,-0.57206)-- (0.57206,-0.41563);
	\draw [color=zzttqq] (0.57206,-0.41563)-- (0.41563,0.57206);
	\draw [color=zzttqq] (0.41563,0.57206)-- (-0.57206,0.41563);
	\draw [color=zzttqq] (-0.57206,-0.41563)-- (0.41563,-0.57206);
	\draw [color=zzttqq] (0.41563,-0.57206)-- (0.57206,0.41563);
	\draw [color=zzttqq] (0.57206,0.41563)-- (-0.41563,0.57206);
	\draw [color=zzttqq] (-0.41563,0.57206)-- (-0.57206,-0.41563);
	\draw [color=zzttqq] (0.21851,-0.6725)-- (0.6725,0.21851);
	\draw [color=zzttqq] (0.6725,0.21851)-- (-0.21851,0.6725);
	\draw [color=zzttqq] (-0.21851,0.6725)-- (-0.6725,-0.21851);
	\draw [color=zzttqq] (-0.6725,-0.21851)-- (0.21851,-0.6725);
	\draw [color=zzttqq] (0.73711,0)-- (0.03,0.70711);
	\draw [color=zzttqq] (0.03,0.70711)-- (-0.67711,0);
	\draw [color=zzttqq] (-0.67711,0)-- (0.03,-0.70711);
	\draw [color=zzttqq] (0.03,-0.70711)-- (0.73711,0);
	\draw [color=zzttqq] (0.22778,0.70103)-- (-0.66323,0.24704);
	\draw [color=zzttqq] (-0.66323,0.24704)-- (-0.20924,-0.64397);
	\draw [color=zzttqq] (-0.20924,-0.64397)-- (0.68177,-0.18998);
	\draw [color=zzttqq] (0.68177,-0.18998)-- (0.22778,0.70103);
	\draw [color=zzttqq] (-0.59633,0.43326)-- (-0.4399,-0.55443);
	\draw [color=zzttqq] (-0.4399,-0.55443)-- (0.54779,-0.39799);
	\draw [color=zzttqq] (0.54779,-0.39799)-- (0.39136,0.58969);
	\draw [color=zzttqq] (0.39136,0.58969)-- (-0.59633,0.43326);
	\draw [color=zzttqq] (-0.59633,-0.43326)-- (0.39136,-0.58969);
	\draw [color=zzttqq] (0.39136,-0.58969)-- (0.54779,0.39799);
	\draw [color=zzttqq] (0.54779,0.39799)-- (-0.4399,0.55443);
	\draw [color=zzttqq] (-0.4399,0.55443)-- (-0.59633,-0.43326);
	\draw [color=zzttqq] (0.22778,-0.70103)-- (0.68177,0.18998);
	\draw [color=zzttqq] (0.68177,0.18998)-- (-0.20924,0.64397);
	\draw [color=zzttqq] (-0.20924,0.64397)-- (-0.66323,-0.24704);
	\draw [color=zzttqq] (-0.66323,-0.24704)-- (0.22778,-0.70103);
	\draw [color=zzttqq] (0.29249,-0.57404)-- (0.22778,-0.70103);
	\draw [color=zzttqq] (0.22778,-0.70103)-- (0.10078,-0.63632);
	\draw [color=zzttqq] (0.10078,-0.63632)-- (0.03,-0.70711);
	\draw [color=zzttqq] (0.03,-0.70711)-- (-0.09258,-0.58453);
	\draw [color=zzttqq] (-0.09258,-0.58453)-- (-0.20924,-0.64397);
	\draw [color=zzttqq] (-0.20924,-0.64397)-- (-0.26868,-0.52731);
	\draw [color=zzttqq] (-0.26868,-0.52731)-- (-0.4399,-0.55443);
	\draw [color=zzttqq] (-0.4399,-0.55443)-- (-0.45556,-0.45556);
	\draw [color=zzttqq] (-0.45556,-0.45556)-- (-0.59633,-0.43326);
	\draw [color=zzttqq] (-0.59633,-0.43326)-- (-0.57404,-0.29249);
	\draw [color=zzttqq] (-0.57404,-0.29249)-- (-0.66323,-0.24704);
	\draw [color=zzttqq] (-0.66323,-0.24704)-- (-0.58453,-0.09258);
	\draw [color=zzttqq] (-0.58453,-0.09258)-- (-0.67711,0);
	\draw [color=zzttqq] (-0.67711,0)-- (-0.58453,0.09258);
	\draw [color=zzttqq] (-0.58453,0.09258)-- (-0.66323,0.24704);
	\draw [color=zzttqq] (-0.66323,0.24704)-- (-0.57404,0.29249);
	\draw [color=zzttqq] (-0.57404,0.29249)-- (-0.59633,0.43326);
	\draw [color=zzttqq] (-0.59633,0.43326)-- (-0.45556,0.45556);
	\draw [color=zzttqq] (-0.45556,0.45556)-- (-0.4399,0.55443);
	\draw [color=zzttqq] (-0.4399,0.55443)-- (-0.26868,0.52731);
	\draw [color=zzttqq] (-0.26868,0.52731)-- (-0.20924,0.64397);
	\draw [color=zzttqq] (-0.20924,0.64397)-- (-0.09258,0.58453);
	\draw [color=zzttqq] (-0.09258,0.58453)-- (0.03,0.70711);
	\draw [color=zzttqq] (0.03,0.70711)-- (0.10078,0.63632);
	\draw [color=zzttqq] (0.10078,0.63632)-- (0.22778,0.70103);
	\draw [color=zzttqq] (0.22778,0.70103)-- (0.29249,0.57404);
	\draw [color=zzttqq] (0.29249,0.57404)-- (0.39136,0.58969);
	\draw [color=zzttqq] (0.39136,0.58969)-- (0.41848,0.41848);
	\draw [color=zzttqq] (0.41848,0.41848)-- (0.54779,0.39799);
	\draw [color=zzttqq] (0.54779,0.39799)-- (0.52731,0.26868);
	\draw [color=zzttqq] (0.52731,0.26868)-- (0.68177,0.18998);
	\draw [color=zzttqq] (0.68177,0.18998)-- (0.63632,0.10078);
	\draw [color=zzttqq] (0.63632,0.10078)-- (0.73711,0);
	\draw [color=zzttqq] (0.73711,0)-- (0.63632,-0.10078);
	\draw [color=zzttqq] (0.63632,-0.10078)-- (0.68177,-0.18998);
	\draw [color=zzttqq] (0.68177,-0.18998)-- (0.52731,-0.26868);
	\draw [color=zzttqq] (0.52731,-0.26868)-- (0.54779,-0.39799);
	\draw [color=zzttqq] (0.54779,-0.39799)-- (0.41848,-0.41848);
	\draw [color=zzttqq] (0.41848,-0.41848)-- (0.39136,-0.58969);
	\draw [color=zzttqq] (0.39136,-0.58969)-- (0.29249,-0.57404);
	\draw [color=zzttqq] (0.29249,-0.57404)-- (0.29249,-0.57404);
	\draw [color=zzttqq] (0.29249,-0.57404)-- (0.22778,-0.70103);
	\draw [color=zzttqq] (0.22778,-0.70103)-- (0.10078,-0.63632);
	\draw [color=zzttqq] (0.10078,-0.63632)-- (0.03,-0.70711);
	\draw [color=zzttqq] (0.03,-0.70711)-- (-0.09258,-0.58453);
	\draw [color=zzttqq] (-0.09258,-0.58453)-- (-0.20924,-0.64397);
	\draw [color=zzttqq] (-0.20924,-0.64397)-- (-0.26868,-0.52731);
	\draw [color=zzttqq] (-0.26868,-0.52731)-- (-0.4399,-0.55443);
	\draw [color=zzttqq] (-0.4399,-0.55443)-- (-0.45556,-0.45556);
	\draw [color=zzttqq] (-0.45556,-0.45556)-- (-0.59633,-0.43326);
	\draw [color=zzttqq] (-0.59633,-0.43326)-- (-0.57404,-0.29249);
	\draw [color=zzttqq] (-0.57404,-0.29249)-- (-0.66323,-0.24704);
	\draw [color=zzttqq] (-0.66323,-0.24704)-- (-0.58453,-0.09258);
	\draw [color=zzttqq] (-0.58453,-0.09258)-- (-0.67711,0);
	\draw [color=zzttqq] (-0.67711,0)-- (-0.58453,0.09258);
	\draw [color=zzttqq] (-0.58453,0.09258)-- (-0.66323,0.24704);
	\draw [color=zzttqq] (-0.66323,0.24704)-- (-0.57404,0.29249);
	\draw [color=zzttqq] (-0.57404,0.29249)-- (-0.59633,0.43326);
	\draw [color=zzttqq] (-0.59633,0.43326)-- (-0.45556,0.45556);
	\draw [color=zzttqq] (-0.45556,0.45556)-- (-0.4399,0.55443);
	\draw [color=zzttqq] (-0.4399,0.55443)-- (-0.26868,0.52731);
	\draw [color=zzttqq] (-0.26868,0.52731)-- (-0.20924,0.64397);
	\draw [color=zzttqq] (-0.20924,0.64397)-- (-0.09258,0.58453);
	\draw [color=zzttqq] (-0.09258,0.58453)-- (0.03,0.70711);
	\draw [color=zzttqq] (0.03,0.70711)-- (0.10078,0.63632);
	\draw [color=zzttqq] (0.10078,0.63632)-- (0.22778,0.70103);
	\draw [color=zzttqq] (0.22778,0.70103)-- (0.29249,0.57404);
	\draw [color=zzttqq] (0.29249,0.57404)-- (0.39136,0.58969);
	\draw [color=zzttqq] (0.39136,0.58969)-- (0.41848,0.41848);
	\draw [color=zzttqq] (0.41848,0.41848)-- (0.54779,0.39799);
	\draw [color=zzttqq] (0.54779,0.39799)-- (0.52731,0.26868);
	\draw [color=zzttqq] (0.52731,0.26868)-- (0.68177,0.18998);
	\draw [color=zzttqq] (0.68177,0.18998)-- (0.63632,0.10078);
	\draw [color=zzttqq] (0.63632,0.10078)-- (0.73711,0);
	\draw [color=zzttqq] (0.73711,0)-- (0.63632,-0.10078);
	\draw [color=zzttqq] (0.63632,-0.10078)-- (0.68177,-0.18998);
	\draw [color=zzttqq] (0.68177,-0.18998)-- (0.52731,-0.26868);
	\draw [color=zzttqq] (0.52731,-0.26868)-- (0.54779,-0.39799);
	\draw [color=zzttqq] (0.54779,-0.39799)-- (0.41848,-0.41848);
	\draw [color=zzttqq] (0.41848,-0.41848)-- (0.39136,-0.58969);
	\draw [color=zzttqq] (0.39136,-0.58969)-- (0.29249,-0.57404);
	\draw [color=zzttqq] (0.29249,-0.57404)-- (0.29249,-0.57404);
	\draw [color=zzttqq] (0.29249,-0.57404)-- (0.22778,-0.70103);
	\draw [color=zzttqq] (0.22778,-0.70103)-- (0.10078,-0.63632);
	\draw [color=zzttqq] (0.10078,-0.63632)-- (0.03,-0.70711);
	\draw [color=zzttqq] (0.03,-0.70711)-- (-0.09258,-0.58453);
	\draw [color=zzttqq] (-0.09258,-0.58453)-- (-0.20924,-0.64397);
	\draw [color=zzttqq] (-0.20924,-0.64397)-- (-0.26868,-0.52731);
	\draw [color=zzttqq] (-0.26868,-0.52731)-- (-0.4399,-0.55443);
	\draw [color=zzttqq] (-0.4399,-0.55443)-- (-0.45556,-0.45556);
	\draw [color=zzttqq] (-0.45556,-0.45556)-- (-0.59633,-0.43326);
	\draw [color=zzttqq] (-0.59633,-0.43326)-- (-0.57404,-0.29249);
	\draw [color=zzttqq] (-0.57404,-0.29249)-- (-0.66323,-0.24704);
	\draw [color=zzttqq] (-0.66323,-0.24704)-- (-0.58453,-0.09258);
	\draw [color=zzttqq] (-0.58453,-0.09258)-- (-0.67711,0);
	\draw [color=zzttqq] (-0.67711,0)-- (-0.58453,0.09258);
	\draw [color=zzttqq] (-0.58453,0.09258)-- (-0.66323,0.24704);
	\draw [color=zzttqq] (-0.66323,0.24704)-- (-0.57404,0.29249);
	\draw [color=zzttqq] (-0.57404,0.29249)-- (-0.59633,0.43326);
	\draw [color=zzttqq] (-0.59633,0.43326)-- (-0.45556,0.45556);
	\draw [color=zzttqq] (-0.45556,0.45556)-- (-0.4399,0.55443);
	\draw [color=zzttqq] (-0.4399,0.55443)-- (-0.26868,0.52731);
	\draw [color=zzttqq] (-0.26868,0.52731)-- (-0.20924,0.64397);
	\draw [color=zzttqq] (-0.20924,0.64397)-- (-0.09258,0.58453);
	\draw [color=zzttqq] (-0.09258,0.58453)-- (0.03,0.70711);
	\draw [color=zzttqq] (0.03,0.70711)-- (0.10078,0.63632);
	\draw [color=zzttqq] (0.10078,0.63632)-- (0.22778,0.70103);
	\draw [color=zzttqq] (0.22778,0.70103)-- (0.29249,0.57404);
	\draw [color=zzttqq] (0.29249,0.57404)-- (0.39136,0.58969);
	\draw [color=zzttqq] (0.39136,0.58969)-- (0.41848,0.41848);
	\draw [color=zzttqq] (0.41848,0.41848)-- (0.54779,0.39799);
	\draw [color=zzttqq] (0.54779,0.39799)-- (0.52731,0.26868);
	\draw [color=zzttqq] (0.52731,0.26868)-- (0.68177,0.18998);
	\draw [color=zzttqq] (0.68177,0.18998)-- (0.63632,0.10078);
	\draw [color=zzttqq] (0.63632,0.10078)-- (0.73711,0);
	\draw [color=zzttqq] (0.73711,0)-- (0.63632,-0.10078);
	\draw [color=zzttqq] (0.63632,-0.10078)-- (0.68177,-0.18998);
	\draw [color=zzttqq] (0.68177,-0.18998)-- (0.52731,-0.26868);
	\draw [color=zzttqq] (0.52731,-0.26868)-- (0.54779,-0.39799);
	\draw [color=zzttqq] (0.54779,-0.39799)-- (0.41848,-0.41848);
	\draw [color=zzttqq] (0.41848,-0.41848)-- (0.39136,-0.58969);
	\draw [color=zzttqq] (0.39136,-0.58969)-- (0.29249,-0.57404);
	\draw [color=zzttqq] (0.29249,-0.57404)-- (0.29249,-0.57404);
	\begin{scriptsize}
	\draw [fill=black] (0.70711,0) circle (1.5pt);
	\draw[color=black] (0.67,0.015) node {$A_1$};
	\draw [fill=black] (0,0.70711) circle (1.5pt);
	\draw[color=black] (-0.035,0.70976) node {$A_2$};
	\draw [fill=black] (-0.70711,0) circle (1.5pt);
	\draw[color=black] (-0.72,0.026) node {$A_3$};
	\draw [fill=black] (0,-0.70711) circle (1.5pt);
	\draw[color=black] (-0.035,-0.70976) node {$A_4$};
	\draw [fill=black] (0.21851,0.6725) circle (1.5pt);
	\draw[color=black] (0.216,0.64) node {$B_1$};
	\draw [fill=black] (-0.6725,0.21851) circle (1.5pt);
	\draw[color=black] (-0.7,0.205) node {$B_2$};
	\draw [fill=black] (-0.21851,-0.6725) circle (1.5pt);
	\draw[color=black] (-0.19169,-0.7) node {$B_3$};
	\draw [fill=black] (0.6725,-0.21851) circle (1.5pt);
	\draw[color=black] (0.70087,-0.24) node {$B_4$};
	\draw [fill=black] (-0.57206,0.41563) circle (1.5pt);
	\draw[color=black] (-0.54,0.4) node {$C_1$};
	\draw [fill=black] (-0.41563,-0.57206) circle (1.5pt);
	\draw[color=black] (-0.38899,-0.595) node {$C_2$};
	\draw [fill=black] (0.57206,-0.41563) circle (1.5pt);
	\draw[color=black] (0.61,-0.425) node {$C_3$};
	\draw [fill=black] (0.41563,0.57206) circle (1.5pt);
	\draw[color=black] (0.455,0.56) node {$C_4$};
	\draw [fill=black] (-0.57206,-0.41563) circle (1.5pt);1	\draw[color=black] (-0.54245,-0.395) node {$D_1$};
	\draw [fill=black] (0.41563,-0.57206) circle (1.5pt);
	\draw[color=black] (0.455,-0.579) node {$D_2$};
	\draw [fill=black] (0.57206,0.41563) circle (1.5pt);
	\draw[color=black] (0.60065,0.44) node {$D_3$};
	\draw [fill=black] (-0.41563,0.57206) circle (1.5pt);
	\draw[color=black] (-0.38899,0.6) node {$D_4$};
	\draw [fill=black] (0.21851,-0.6725) circle (1.5pt);
	\draw[color=black] (0.21,-0.62) node {$E_1$};
	\draw [fill=black] (0.6725,0.21851) circle (1.5pt);
	\draw[color=black] (0.71,0.24) node {$E_2$};
	\draw [fill=black] (-0.21851,0.6725) circle (1.5pt);
	\draw[color=black] (-0.19169,0.6976) node {$E_3$};
	\draw [fill=black] (-0.6725,-0.21851) circle (1.5pt);
	\draw[color=black] (-0.69,-0.185) node {$E_4$};
	\draw [fill=black] (0.41563,-0.57206) circle (1.5pt);
	\draw [fill=black] (0.73711,0) circle (1.5pt);
	\draw[color=black] (0.762,0.025) node {$A'_1$};
	\draw [fill=black] (0.03,0.70711) circle (1.5pt);
	\draw[color=black] (0.06825,0.70976) node {$A'_2$};
	\draw [fill=black] (-0.67711,0) circle (1.5pt);
	\draw[color=black] (-0.63954,0.01209) node {$A'_3$};
	\draw [fill=black] (0.03,-0.70711) circle (1.5pt);
	\draw[color=black] (0.06825,-0.70976) node {$A'_4$};
	\draw [fill=black] (0.22778,0.70103) circle (1.5pt);
	\draw[color=black] (0.26555,0.70976) node {$B'_1$};
	\draw [fill=black] (-0.66323,0.24704) circle (1.5pt);
	\draw[color=black] (-0.7,0.27131) node {$B'_2$};
	\draw [fill=black] (-0.20924,-0.64397) circle (1.5pt);
	\draw[color=black] (-0.19169,-0.6) node {$B'_3$};
	\draw [fill=black] (0.68177,-0.18998) circle (1.5pt);
	\draw[color=black] (0.7,-0.1582) node {$B'_4$};
	\draw [fill=black] (-0.59633,0.43326) circle (1.5pt);
	\draw[color=black] (-0.63,0.45) node {$C'_1$};
	\draw [fill=black] (-0.4399,-0.55443) circle (1.5pt);
	\draw[color=black] (-0.475,-0.545) node {$C'_2$};
	\draw [fill=black] (0.54779,-0.39799) circle (1.5pt);
	\draw[color=black] (0.52,-0.365) node {$C'_3$};
	\draw [fill=black] (0.39136,0.58969) circle (1.5pt);
	\draw[color=black] (0.38,0.625) node {$C'_4$};
	\draw [fill=black] (0.52731,-0.26868) circle (1.5pt);
	\draw[color=black] (0.504,-0.26) node {$V$};
	\draw [fill=black] (-0.59633,-0.43326) circle (1.5pt);
	\draw[color=black] (-0.618,-0.468) node {$D'_1$};
	\draw [fill=black] (0.39136,-0.58969) circle (1.5pt);
	\draw[color=black] (0.38,-0.63) node {$D'_2$};
	\draw [fill=black] (0.54779,0.39799) circle (1.5pt);
	\draw[color=black] (0.51,0.37) node {$D'_3$};
	\draw [fill=black] (-0.4399,0.55443) circle (1.5pt);
	\draw[color=black] (-0.48,0.57) node {$D'_4$};
	\draw [fill=black] (0.22778,-0.70103) circle (1.5pt);
	\draw[color=black] (0.26555,-0.7) node {$E'_1$};
	\draw [fill=black] (0.68177,0.18998) circle (1.5pt);
	\draw[color=black] (0.715,0.175) node {$E'_2$};
	\draw [fill=black] (-0.20924,0.64397) circle (1.5pt);
	\draw[color=black] (-0.195,0.605) node {$E'_3$};
	\draw [fill=black] (-0.66323,-0.24704) circle (1.5pt);
	\draw[color=black] (-0.69,-0.28) node {$E'_4$};
	\draw [fill=black] (0.66249,0.19887) circle (1.5pt);
	\draw[color=black] (0.63,0.19) node {$Q$};
	\draw [fill=black] (0.61043,0.09668) circle (1.5pt);
	\draw[color=black] (0.575,0.1) node {$R$};
	\draw [fill=black] (0.63632,0.10078) circle (1.5pt);
	\draw[color=black] (0.66,0.11) node {$P$};
	\draw [fill=black] (0.55067,-0.28058) circle (1.5pt);
	\draw[color=black] (0.57247,-0.3) node {$T$};
	\draw [fill=black] (0.66361,-0.20106) circle (1.5pt);
	\draw[color=black] (0.638,-0.19) node {$U$};
	\end{scriptsize}
\end{tikzpicture}

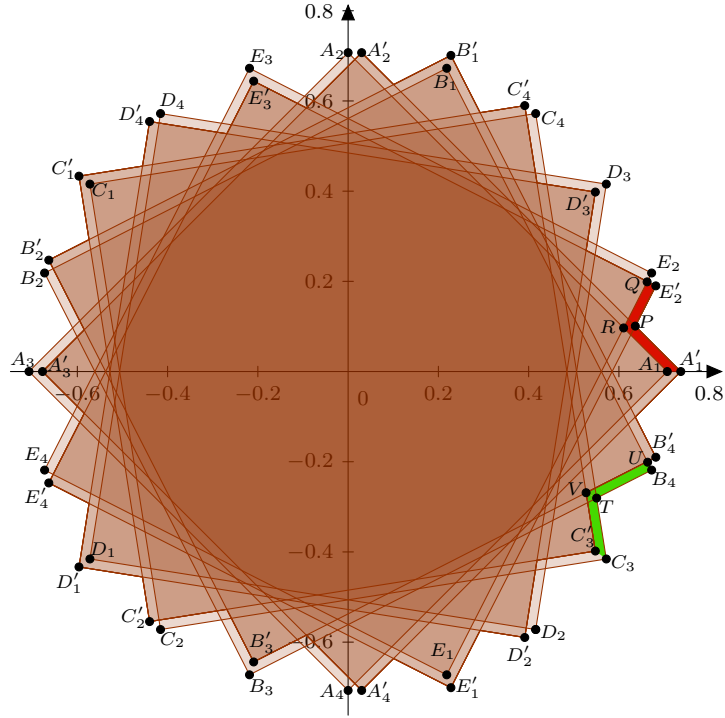
\captionof{figure}{The shifted $(5,4)$-setup}
   \label{negyz}
\end{center}

\subsection{A counterexample of 4 regular triangles}
\begin{theorem}
There exists a polygon that is the union of $4$ regular triangles with side length $1$ for which the perimeter-to-area ratio is greater then $4\sqrt{3}$ (which is the perimeter-to-area ratio of a single triangle).
\end{theorem}
\begin{proof}
The idea is the similar as before, we start with a basic $(4,3)$-setup. In this case the calculation is particularly convenient if we translate solely two of the triangles, since the translation will not change the perimeter of the union, however the area will decrease.

For the sake of exactness we prove the theorem through three easy lemmas.

\begin{lemma}
Suppose that we have 4 triangles obtained by pattern preserving translations of the basic setup. Then if we apply a pattern preserving regular translation to one of the triangles then it does not change the perimeter of the union.
\end{lemma}
\begin{proof}
Without loss of generality we can assume that we translate a triangle to the positive direction of the $y$ axis, as in Figure \ref{ker}. 

\definecolor{ccqqcc}{rgb}{0.8,0,0.8}
\definecolor{uququq}{rgb}{0.25098,0.25098,0.25098}
\definecolor{qqcctt}{rgb}{0,0.8,0.2}
\definecolor{qqqqff}{rgb}{0,0,1}
\definecolor{zzttqq}{rgb}{0.6,0.2,0}

\begin{tikzpicture}[line cap=round,line join=round,>=triangle 45,x=7.0cm,y=7.0cm]
\clip(-0.88885,-0.31678) rectangle (0.88127,0.74378);
\fill[color=zzttqq,fill=zzttqq,fill opacity=0.25] (0.59735,0) -- (-0.26868,0.5) -- (-0.26868,-0.5) -- cycle;
\fill[color=zzttqq,fill=zzttqq,fill opacity=0.2] (-0.57735,0) -- (0.28868,-0.5) -- (0.28868,0.5) -- cycle;
\fill[color=zzttqq,fill=zzttqq,fill opacity=0.25] (0,-0.57735) -- (0.5,0.28868) -- (-0.5,0.28868) -- cycle;
\fill[color=zzttqq,fill=zzttqq,fill opacity=0.1] (-0.5,-0.18868) -- (0.5,-0.18868) -- (0,0.67735) -- cycle;
\draw [color=zzttqq] (0.59735,0)-- (-0.26868,0.5);
\draw [color=zzttqq] (-0.26868,0.5)-- (-0.26868,-0.5);
\draw [color=zzttqq] (-0.26868,-0.5)-- (0.59735,0);
\draw [color=qqcctt] (0,0.59735)-- (-0.5,-0.26868);
\draw [color=qqcctt] (-0.5,-0.26868)-- (0.5,-0.26868);
\draw [color=qqcctt] (0.5,-0.26868)-- (0,0.59735);
\draw [color=zzttqq] (-0.57735,0)-- (0.28868,-0.5);
\draw [color=zzttqq] (0.28868,-0.5)-- (0.28868,0.5);
\draw [color=zzttqq] (0.28868,0.5)-- (-0.57735,0);
\draw [color=zzttqq] (0,-0.57735)-- (0.5,0.28868);
\draw [color=zzttqq] (0.5,0.28868)-- (-0.5,0.28868);
\draw [color=zzttqq] (-0.5,0.28868)-- (0,-0.57735);
\draw [color=zzttqq] (-0.5,-0.18868)-- (0.5,-0.18868);
\draw [color=zzttqq] (0.5,-0.18868)-- (0,0.67735);
\draw [color=zzttqq] (0,0.67735)-- (-0.5,-0.18868);
\draw [line width=1.2pt] (0.408,-0.10932)-- (0.44264,-0.08932);
\draw [line width=1.2pt] (0.11432,0.39934)-- (0.14896,0.41934);
\draw [line width=1.2pt,color=ccqqcc] (0,0.67735)-- (-0.03464,0.61735);
\draw [line width=1.2pt,color=ccqqcc] (0,0.67735)-- (0.03464,0.61735);
\draw [line width=1.2pt] (-0.14396,0.428)-- (-0.10932,0.408);
\draw [line width=1.2pt] (-0.43764,-0.08066)-- (-0.403,-0.10066);
\draw [line width=1.2pt] (-0.46536,-0.20868)-- (-0.5,-0.26868);
\draw [line width=1.2pt] (0.46536,-0.20868)-- (0.5,-0.26868);
\draw [line width=1.2pt] (-0.5,-0.18868)-- (-0.46536,-0.20868);
\draw [line width=1.2pt,color=ccqqcc] (-0.26868,-0.18868)-- (-0.26868,-0.26868);
\draw [line width=1.2pt,color=ccqqcc] (0.28868,-0.18868)-- (0.28868,-0.26868);
\draw [line width=1.2pt,color=ccqqcc] (-0.5,-0.18868)-- (-0.5,-0.26868);
\begin{scriptsize}
\fill [color=black] (0,0.59735) circle (1pt);
\draw[color=black] (0.0094,0.6154) node {$B_0$};
\fill [color=black] (-0.5,-0.26868) circle (1pt);
\draw[color=black] (-0.53,-0.28) node {$B_1$};
\fill [color=black] (0.5,-0.26868) circle (1pt);
\draw[color=black] (0.53,-0.28) node {$B_2$};
\fill [color=black] (0,0.67735) circle (1pt);
\draw[color=black] (0.0094,0.697) node {$B'_0$};
\fill [color=black] (-0.5,-0.18868) circle (1pt);
\draw[color=black] (-0.53,-0.17327) node {$B'_1$};
\fill [color=black] (0.5,-0.18868) circle (1pt);
\draw[color=black] (0.53,-0.17327) node {$B'_2$};
\fill [color=black] (-0.46536,-0.20868) circle (1pt);
\draw[color=black] (-0.435,-0.21) node {$I_4$};
\fill [color=black] (0.46536,-0.20868) circle (1pt);
\draw[color=black] (0.435,-0.21) node {$I_5$};
\fill [color=black] (-0.03464,0.61735) circle (1pt);
\draw[color=black] (-0.06,0.63294) node {$J_0$};
\fill [color=black] (0.03464,0.61735) circle (1pt);
\draw[color=black] (0.059,0.63294) node {$J_1$};
\fill [color=black] (-0.14396,0.428) circle (1pt);
\draw[color=black] (-0.152,0.46) node {$I'_0$};
\fill [color=black] (-0.10932,0.408) circle (1pt);
\draw[color=black] (-0.07,0.41) node {$I_0$};
\fill [color=black] (-0.43764,-0.08066) circle (1pt);
\draw[color=black] (-0.47,-0.09) node {$I'_2$};
\fill [color=black] (-0.403,-0.10066) circle (1pt);
\draw[color=black] (-0.38796,-0.13) node {$I_2$};
\fill [color=black] (0.14896,0.41934) circle (1pt);
\draw[color=black] (0.155,0.46) node {$I'_1$};
\fill [color=black] (0.11432,0.39934) circle (1pt);
\draw[color=black] (0.085,0.41) node {$I_1$};
\fill [color=black] (0.44264,-0.08932) circle (1pt);
\draw[color=black] (0.479,-0.09) node {$I'_3$};
\fill [color=black] (0.408,-0.10932) circle (1pt);
\draw[color=black] (0.405,-0.13) node {$I_3$};
\fill [color=black] (-0.26868,-0.18868) circle (1pt);
\draw[color=black] (-0.29,-0.21) node {$J'_2$};
\fill [color=black] (-0.26868,-0.26868) circle (1pt);
\draw[color=black] (-0.29,-0.29) node {$J_2$};
\fill [color=black] (0.28868,-0.18868) circle (1pt);
\draw[color=black] (0.31,-0.21) node {$J'_3$};
\fill [color=black] (0.28868,-0.26868) circle (1pt);
\draw[color=black] (0.31,-0.29) node {$J_3$};

\end{scriptsize}
\end{tikzpicture}

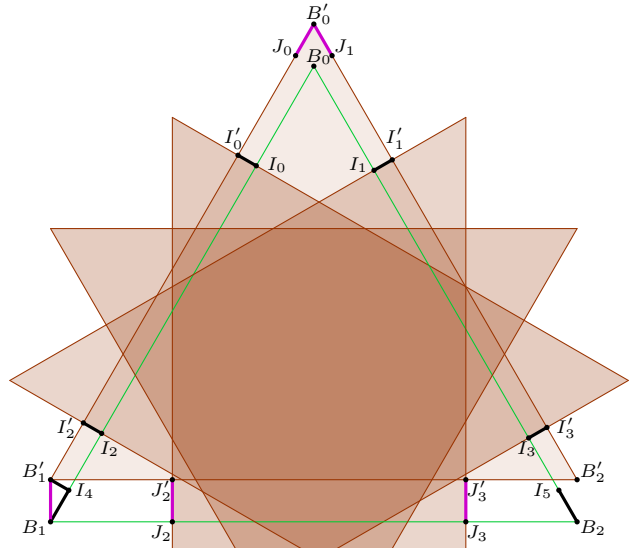
\captionof{figure}{Calculation of the perimeter}
   \label{ker}

Clearly, the perimeter of the union decreases with the length of $I_0I'_0,\dots,I_3I'_3$ and $I_4B_1$ and $I_5B_2$, where $I_4$ and $I_5$ are assigned so that the triangles $B'_1I_4B_1$ and $B'_2I_5B_2$ are right-angled.
The increase of the perimeter is equal to the sum of the lengths of $J_0B'_0$, $J_1B'_0$ and $J_2J'_2,J_3J'_3$. 

Now notice that the triangle $B'_1I_4B_1$ has angles $60^\circ,90^\circ,30^\circ$ respectively and $|B'_1I_4|=|I_0I'_0|=|I_2I'_2|$, 
and $|B'_1B_1|=|J_2J'_2|$. Therefore $|J_2J'_2|=|I_0I'_0|+|I_2I'_2|$. The triangles $B'_1I_4B_1$ and $B_0J_0B'_0$ are congruent, so $|I_4B_1|=|J_0B'_0|$. Similarly $|J_3J'_3|=|I_1I'_1|+|I_3I'_3|$ and $|I_5B_2|=|J_1B'_0|$.

Thus, the increase equals to the decrease so the translation does not change the perimeter.
   \end{proof}
\begin{lemma}
Suppose that we have 4 triangles in the basic $(4,3)$-setup. If we apply a pattern preserving regular translation to one of the triangles then it does not change the area of the union.
\end{lemma}
\begin{proof}
Again, without loss of generality we can assume that we have translated the triangle $B_0B_1B_2$ with the vector $(0,\varepsilon)$ and we obtain the triangle $B'_0B'_1B'_2$ as shown in Figure \ref{ter}. According to the notation of the figure, it is clear that the change in the area is equal to $2(a(B_0B'_0I'_0I_0)+a(I'_2I_2L_1B'_1)-a(L_1J'_2J_2B_1))$ (i.e., the area of the red and pink figures is added and the area of the green figures is subtracted). Let us denote $|B_1J_2|$ by $d$. Now we calculate this value. 

Clearly, \[a(L_1J'_2J_2B_1)=\frac{d+d-\varepsilon/\sqrt{3}}{2}\cdot\varepsilon\]
and by symmetry $|B_1I_2|=|B_0I_0|=d$ and we have $|I'_2I_2|=|I'_0I_0|=\varepsilon/2$.

\[a(B_0B'_0I'_0I_0)=\frac{d+d+(\sqrt{3}/2)\varepsilon}{2}\cdot\frac{\varepsilon}{2}\]

and 
\[a(I'_2I_2L_1B'_1)=\frac{d-(2/\sqrt{3})\varepsilon+d-\varepsilon/(2\sqrt{3})}{2}\cdot\frac{\varepsilon}{2}.\]

Adding up these equalities gives the lemma.

\end{proof}

\definecolor{ccqqcc}{rgb}{0.8,0,0.8}
\definecolor{qqqqcc}{rgb}{0,0,0.8}
\definecolor{ttffqq}{rgb}{0.2,1,0}
\definecolor{ffqqqq}{rgb}{1,0,0}
\definecolor{uququq}{rgb}{0.25098,0.25098,0.25098}
\definecolor{zzttqq}{rgb}{0.6,0.2,0}
\definecolor{qqcctt}{rgb}{0,0.8,0.2}
\definecolor{ffttqq}{rgb}{1,0.2,0}
\definecolor{qqqqff}{rgb}{0,0,1}
\begin{tikzpicture}[line cap=round,line join=round,>=triangle 45,x=8.0cm,y=8.0cm]

\clip(-0.77,-0.33) rectangle (0.77,0.6573);
\fill[color=zzttqq,fill=zzttqq,fill opacity=0.2] (-0.57735,0) -- (0.28868,-0.5) -- (0.28868,0.5) -- cycle;
\fill[color=zzttqq,fill=zzttqq,fill opacity=0.25] (0,-0.57735) -- (0.5,0.28868) -- (-0.5,0.28868) -- cycle;
\fill[color=zzttqq,fill=zzttqq,fill opacity=0.1] (-0.5,-0.2287) -- (0.5,-0.2287) -- (0,0.6374) -- cycle;
\fill[color=ffqqqq,fill=ffqqqq,fill opacity=1.0] (-0.10568,0.39435) -- (-0.13165,0.40935) -- (0,0.6374) -- (0,0.5774) -- cycle;
\fill[color=ffqqqq,fill=ffqqqq,fill opacity=0.75] (0,0.5774) -- (0.10568,0.39435) -- (0.13166,0.40935) -- (0,0.6374) -- cycle;
\fill[color=ffqqqq,fill=ffqqqq,fill opacity=0.75] (-0.46536,-0.2287) -- (-0.5,-0.2287) -- (-0.42031,-0.09067) -- (-0.39433,-0.10566) -- cycle;
\fill[color=ttffqq,fill=ttffqq,fill opacity=0.75] (-0.2887,-0.2887) -- (-0.5,-0.2887) -- (-0.46536,-0.2287) -- (-0.28869,-0.2287) -- cycle;
\fill[color=ttffqq,fill=ttffqq,fill opacity=0.75] (0.5,-0.2887) -- (0.28868,-0.2887) -- (0.28868,-0.2287) -- (0.46536,-0.2287) -- cycle;
\fill[color=ffqqqq,fill=ffqqqq,fill opacity=0.75] (0.5,-0.2287) -- (0.46536,-0.2287) -- (0.39434,-0.10568) -- (0.42032,-0.09068) -- cycle;
\fill[line width=2.4pt,color=qqqqcc,fill=qqqqcc,fill opacity=0.75] (-0.2287,-0.2887) -- (-0.22869,-0.2287) -- (-0.28869,-0.2287) -- (-0.2887,-0.2887) -- cycle;
\fill[line width=3.2pt,color=ccqqcc,fill=ccqqcc,fill opacity=1.0] (0.42032,-0.09068) -- (0.39434,-0.10568) -- (0.40934,-0.13166) -- (0.43532,-0.11666) -- cycle;
\fill[line width=3.2pt,color=ccqqcc,fill=ccqqcc,fill opacity=1.0] (-0.11666,0.43533) -- (-0.13165,0.40935) -- (-0.10568,0.39435) -- (-0.09068,0.42033) -- cycle;
\draw [color=ffttqq] (0.5774,0)-- (-0.28868,0.5);
\draw [color=ffttqq] (-0.28868,0.5)-- (-0.2887,-0.5);
\draw [color=ffttqq] (-0.2887,-0.5)-- (0.5774,0);
\draw [color=qqcctt] (0,0.5774)-- (-0.5,-0.2887);
\draw [color=qqcctt] (-0.5,-0.2887)-- (0.5,-0.2887);
\draw [color=qqcctt] (0.5,-0.2887)-- (0,0.5774);
\draw [color=zzttqq] (-0.57735,0)-- (0.28868,-0.5);
\draw [color=zzttqq] (0.28868,-0.5)-- (0.28868,0.5);
\draw [color=zzttqq] (0.28868,0.5)-- (-0.57735,0);
\draw [color=zzttqq] (0,-0.57735)-- (0.5,0.28868);
\draw [color=zzttqq] (0.5,0.28868)-- (-0.5,0.28868);
\draw [color=zzttqq] (-0.5,0.28868)-- (0,-0.57735);
\draw [color=zzttqq] (-0.5,-0.2287)-- (0.5,-0.2287);
\draw [color=zzttqq] (0.5,-0.2287)-- (0,0.6374);
\draw [color=zzttqq] (0,0.6374)-- (-0.5,-0.2287);
\draw [color=zzttqq] (0.6374,0)-- (-0.22868,0.5);
\draw [color=zzttqq] (-0.22868,0.5)-- (-0.2287,-0.5);
\draw [color=zzttqq] (-0.2287,-0.5)-- (0.6374,0);
\draw [color=ffqqqq] (-0.10568,0.39435)-- (-0.13165,0.40935);
\draw [color=ffqqqq] (-0.13165,0.40935)-- (0,0.6374);
\draw [color=ffqqqq] (0,0.6374)-- (0,0.5774);
\draw [color=ffqqqq] (0,0.5774)-- (-0.10568,0.39435);
\draw [color=ffqqqq] (0,0.5774)-- (0.10568,0.39435);
\draw [color=ffqqqq] (0,0.6374)-- (0,0.5774);
\draw [color=ffqqqq] (-0.5,-0.2287)-- (-0.42031,-0.09067);
\draw [color=ffqqqq] (-0.42031,-0.09067)-- (-0.39433,-0.10566);
\draw [color=ffqqqq] (-0.39433,-0.10566)-- (-0.46536,-0.2287);
\draw [color=ttffqq] (-0.2887,-0.2887)-- (-0.5,-0.2887);
\draw [color=ttffqq] (-0.5,-0.2887)-- (-0.46536,-0.2287);
\draw [color=ttffqq] (-0.28869,-0.2287)-- (-0.2887,-0.2887);
\draw [color=ttffqq] (0.5,-0.2887)-- (0.28868,-0.2887);
\draw [color=ttffqq] (0.28868,-0.2887)-- (0.28868,-0.2287);
\draw [color=ttffqq] (0.28868,-0.2287)-- (0.46536,-0.2287);
\draw [color=ttffqq] (0.46536,-0.2287)-- (0.5,-0.2887);
\draw [color=ffqqqq] (0.5,-0.2287)-- (0.46536,-0.2287);
\draw [color=ffqqqq] (0.46536,-0.2287)-- (0.39434,-0.10568);
\draw [color=ffqqqq] (0.39434,-0.10568)-- (0.42032,-0.09068);
\draw [color=ffqqqq] (0.42032,-0.09068)-- (0.5,-0.2287);

\begin{scriptsize}
\fill [color=black] (0.5774,0) circle (1.5pt);
\draw[color=black] (0.54,0.008) node {$A_0$};
\fill [color=black] (-0.28868,0.5) circle (1.5pt);
\draw[color=black] (-0.27077,0.51532) node {$A_1$};
\fill [color=black] (-0.2887,-0.5) circle (1.5pt);
\draw[color=black] (-0.30119,-0.39725) node {$A_2$};
\fill [color=black] (0,0.5774) circle (1.5pt);
\draw[color=black] (-0.0,0.54) node {$B_0$};
\fill [color=black] (-0.5,-0.2887) circle (1.5pt);
\draw[color=black] (-0.48258,-0.31) node {$B_1$};
\fill [color=black] (0.5,-0.2887) circle (1.5pt);
\draw[color=black] (0.51798,-0.27347) node {$B_2$};
\fill [color=black] (0,0.6374) circle (1.5pt);
\draw[color=black] (0.03,0.64) node {$B'_0$};
\fill [color=black] (-0.5,-0.2287) circle (1.5pt);
\draw[color=black] (-0.524,-0.21279) node {$B'_1$};
\fill [color=black] (0.5,-0.2287) circle (1.5pt);
\draw[color=black] (0.51915,-0.21279) node {$B'_2$};
\fill [color=black] (0.6374,0) circle (1.5pt);
\draw[color=black] (0.67,0.004) node {$A'_0$};
\fill [color=black] (-0.22868,0.5) circle (1.5pt);
\draw[color=black] (-0.20991,0.51532) node {$A'_1$};
\fill [color=black] (-0.2287,-0.5) circle (1.5pt);
\draw[color=black] (-0.25321,-0.39725) node {$A'_2$};
\fill [color=black] (-0.13165,0.40935) circle (1.5pt);
\draw[color=black] (-0.16,0.4) node {$I'_0$};
\fill [color=black] (-0.10568,0.39435) circle (1.5pt);
\draw[color=black] (-0.099,0.37) node {$I_0$};
\fill [color=black] (0.10568,0.39435) circle (1.5pt);
\draw[color=black] (0.083,0.405) node {$I_1$};
\fill [color=black] (0.13166,0.40935) circle (1.5pt);
\draw[color=black] (0.14,0.44) node {$I'_1$};
\fill [color=black] (0.46536,-0.2287) circle (1.5pt);
\draw[color=black] (0.438,-0.21279) node {$L_2$};
\fill [color=black] (0.42032,-0.09068) circle (1.5pt);
\draw[color=black] (0.425,-0.066) node {$I'_3$};
\fill [color=black] (0.39434,-0.10568) circle (1.5pt);
\draw[color=black] (0.375,-0.09023) node {$I_3$};
\fill [color=black] (0.28868,-0.2287) circle (1.5pt);also 
\draw[color=black] (0.30266,-0.21279) node {$J_3$};
\fill [color=black] (-0.28869,-0.2287) circle (1.5pt);
\draw[color=black] (-0.27194,-0.21279) node {$J'_2$};
\fill [color=black] (-0.2887,-0.2887) circle (1.5pt);
\draw[color=black] (-0.27428,-0.31) node {$J_2$};
\fill [color=black] (0.28868,-0.2887) circle (1.5pt);
\draw[color=black] (0.305,-0.31) node {$J'_3$};
\fill [color=black] (-0.42031,-0.09067) circle (1.5pt);
\draw[color=black] (-0.45,-0.1) node {$I'_2$};
\fill [color=black] (-0.39433,-0.10566) circle (1.5pt);
\draw[color=black] (-0.37,-0.1) node {$I_2$};
\fill [color=black] (-0.46536,-0.2287) circle (1.5pt);
\draw[color=black] (-0.435,-0.21279) node {$L_1$};
\fill [color=black] (-0.22869,-0.2287) circle (1.5pt);
\draw[color=black] (-0.21,-0.25) node {$K_2$};
\fill [color=black] (-0.2287,-0.2887) circle (1.5pt);
\draw[color=black] (-0.20991,-0.31) node {$K'_2$};
\fill [color=black] (0.40934,-0.13166) circle (1.5pt);
\draw[color=black] (0.405,-0.16) node {$K_3$};
\fill [color=black] (0.43532,-0.11666) circle (1.5pt);
\draw[color=black] (0.465,-0.122) node {$K'_3$};
\fill [color=black] (-0.11666,0.43533) circle (1.5pt);
\draw[color=black] (-0.13,0.465) node {$K'_0$};
\fill [color=black] (-0.09068,0.42033) circle (1.5pt);
\draw[color=black] (-0.06,0.43) node {$K_0$};
\end{scriptsize}
\end{tikzpicture}

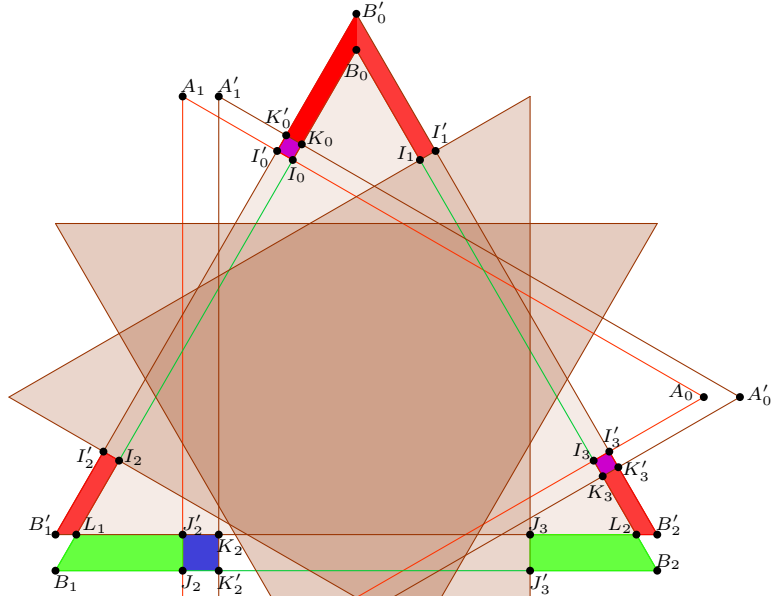
\captionof{figure}{Calculation of the area}
   \label{ter}

\begin{lemma}
 Suppose that we translate away two different neighbouring triangles from the basic setup by pattern preserving regular translations such that the directions of the translations are perpendicular. Then the area of the union of the triangles decreases.
 
\end{lemma}

\begin{proof}
Suppose that we translate the triangles consecutively. By the previous lemma we have that the translation of one triangle does not change the area. This would be true for the 'second' triangle as well, except that after the translation of the 'first' triangle, the picture is already modified. To obtain the difference in the area, it is enough to consider the modifications what the translation of the 'first' triangle gives (the blue and pink rectangles on the figure). In our case it will decrease the area.

To be precise, let the translated triangles be $A_0A_1A_2$ and $B_0B_1B_2$ and the translations are $(\delta,0)$ and $(\varepsilon,0)$ (Figure \ref{ter}). Then suppose that we have translated $A_0A_1A_2$ to 
$A'_0A'_1A'_2$ 'first'. Now the translation of $B_0B_1B_2$ to $B'_0B'_1B'_2$ increases the 
area of the union by

\[a(B'_0B_0K_0K'_0)+a(I'_2I_2L_1B'_1)+a(B'_2L_2K_3K'_3)+a(B'_0B_0I_1I'_1)\]

and decreases it by $a(B_1L_1K_2K'_2)+a(J_3L_2B_2J'_3)$ (the quadrilaterals added coloured by red, the subtracted ones coloured by green and blue on Figure \ref{ter}).

Thus we have
\[a(B'_0B_0K_0K'_0)+a(I'_2I_2L_1B'_1)+a(B'_2L_2K_3K'_3)+a(B'_0B_0I_1I'_1)\]\[-a(B_1L_1K_2K'_2)-a(J_3L_2B_2J'_3)\]\[<a(B'_0B_0I_0I'_0)+a(I'_2I_2L_1B'_1)+a(B'_2L_2I_3I'_3)+A(B'_0B_0I_1I'_1)\]\[-a(B_1L_1J_2J'_2)-a(J_3L_2B_2J'_3)=0.\]

\end{proof}
Putting together the 3 lemmas we have that after translating two different neighbouring triangles by regular pattern preserving translations,
the area decreases, but the perimeter does not change. Since by Lemma \ref{common_centre} the perimeter-to-area ratio of the basic $(4,3)$-setup equals to $4\sqrt{3}$, we are done. 
\end{proof}

\section{Other counterexamples}
\label{s:other}
Let us mention that there exist four squares in the plain forming a counterexample and are 
close to a $(4, 4)$-setup in some sense. Without proof we present a construction of such squares with perimeter-to-area ratio of the 
union $\approx 4.02 > 4$.
The four squares will be the following: 
$$S_1 = \text{conv} \left\{ \left( \frac{1}{2}, \frac{1}{2} \right),
  \left( -\frac{1}{2}, \frac{1}{2} \right),
  \left( -\frac{1}{2}, -\frac{1}{2} \right),
  \left( \frac{1}{2}, -\frac{1}{2} \right) \right\},
$$
$$S_2 = \text{conv} \left\{ \left( \frac{149}{650}, \frac{399}{650} \right),
  \left( -\frac{451}{650}, \frac{149}{650} \right),
  \left( -\frac{201}{650}, -\frac{451}{650} \right),
  \left( \frac{399}{650}, -\frac{201}{650} \right) \right\},
$$
$$S_3 = \text{conv} \left\{ \left( \frac{399}{650}, \frac{201}{650} \right),
  \left( -\frac{201}{650}, \frac{451}{650} \right),
  \left( -\frac{451}{650}, -\frac{149}{650} \right),
  \left( \frac{149}{650}, -\frac{399}{650} \right) \right\},
$$
$$S_4 = \text{conv} \left\{ \left( -\frac{91}{1450}, \frac{41}{58} \right), 
  \left( -\frac{1141}{1450}, \frac{1}{58} \right), 
  \left( -\frac{141}{1450}, -\frac{41}{58} \right), 
  \left( \frac{909}{1450}, -\frac{1}{58} \right) \right\}.
$$

\begin{wrapfigure}{r}{5cm}
	\centering 
	\definecolor{zzttqq}{rgb}{0.6,0.2,0}
	\definecolor{qqqqff}{rgb}{0,0,1}
	\begin{tikzpicture}[line cap=round,line join=round,>=triangle 45,x=3.0cm,y=3.0cm]
	\draw[->,color=black] (-0.55062,0) -- (0.71288,0);
	
	%\draw[shift={(\x,0)},color=black] (0pt,2pt) -- (0pt,-2pt) node[below] {\footnotesize $\x$};
	\draw[->,color=black] (0,-0.67654) -- (0,0.6954);
	
	%\draw[shift={(0,\y)},color=black] (2pt,0pt) -- (-2pt,0pt) node[left] {\footnotesize $\y$};
	%\draw[color=black] (0pt,-10pt) node[right] {\footnotesize $0$};
	\clip(-0.55062,-0.67654) rectangle (0.71288,0.6954);
	\fill[color=zzttqq,fill=zzttqq,fill opacity=0.1] (-0.5,0.23205) -- (0.5,0.23205) -- (0,-0.634) -- cycle;
	\fill[color=zzttqq,fill=zzttqq,fill opacity=0.1] (-0.5,-0.23205) -- (0.5,-0.23205) -- (0,0.634) -- cycle;
	\fill[color=zzttqq,fill=zzttqq,fill opacity=0.1] (-0.23205,-0.5) -- (-0.23205,0.5) -- (0.634,0) -- cycle;
	\fill[color=zzttqq,fill=zzttqq,fill opacity=0.1] (0,0.634) -- (0.23206,0.23205) -- (0.5,0.23205) -- (0.43302,0.11603) -- (0.634,0) -- (0.43302,-0.11603) -- (0.5,-0.23205) -- (0.23206,-0.23205) -- (0,-0.634) -- (-0.11603,-0.43302) -- (-0.23205,-0.5) -- (-0.23205,-0.23205) -- (-0.5,-0.23205) -- (-0.36603,0) -- (-0.5,0.23205) -- (-0.23205,0.23205) -- (-0.23205,0.5) -- (-0.11603,0.43302) -- cycle;
	\draw [color=zzttqq] (-0.5,0.23205)-- (0.5,0.23205);
	\draw [color=zzttqq] (0.5,0.23205)-- (0,-0.634);
	\draw [color=zzttqq] (0,-0.634)-- (-0.5,0.23205);
	\draw [color=zzttqq] (-0.5,-0.23205)-- (0.5,-0.23205);
	\draw [color=zzttqq] (0.5,-0.23205)-- (0,0.634);
	\draw [color=zzttqq] (0,0.634)-- (-0.5,-0.23205);
	\draw [color=zzttqq] (-0.23205,-0.5)-- (-0.23205,0.5);
	\draw [color=zzttqq] (-0.23205,0.5)-- (0.634,0);
	\draw [color=zzttqq] (0.634,0)-- (-0.23205,-0.5);
	\draw [color=zzttqq] (0,0.634)-- (0.23206,0.23205);
	\draw [color=zzttqq] (0.23206,0.23205)-- (0.5,0.23205);
	\draw [color=zzttqq] (0.5,0.23205)-- (0.43302,0.11603);
	\draw [color=zzttqq] (0.43302,0.11603)-- (0.634,0);
	\draw [color=zzttqq] (0.634,0)-- (0.43302,-0.11603);
	\draw [color=zzttqq] (0.43302,-0.11603)-- (0.5,-0.23205);
	\draw [color=zzttqq] (0.5,-0.23205)-- (0.23206,-0.23205);
	\draw [color=zzttqq] (0.23206,-0.23205)-- (0,-0.634);
	\draw [color=zzttqq] (0,-0.634)-- (-0.11603,-0.43302);
	\draw [color=zzttqq] (-0.11603,-0.43302)-- (-0.23205,-0.5);
	\draw [color=zzttqq] (-0.23205,-0.5)-- (-0.23205,-0.23205);
	\draw [color=zzttqq] (-0.23205,-0.23205)-- (-0.5,-0.23205);
	\draw [color=zzttqq] (-0.5,-0.23205)-- (-0.36603,0);
	\draw [color=zzttqq] (-0.36603,0)-- (-0.5,0.23205);
	\draw [color=zzttqq] (-0.5,0.23205)-- (-0.23205,0.23205);
	\draw [color=zzttqq] (-0.23205,0.23205)-- (-0.23205,0.5);
	\draw [color=zzttqq] (-0.23205,0.5)-- (-0.11603,0.43302);
	\draw [color=zzttqq] (-0.11603,0.43302)-- (0,0.634);
	\draw (0.23206,0.23205)-- (0.23206,-0.23205);
	\begin{scriptsize}
	\draw [fill=black] (-0.5,0.23205) circle (1.5pt);
	\draw[color=black] (-0.47268,0.28635) node {$A_1$};
	\draw [fill=black] (0.5,0.23205) circle (1.5pt);
	\draw[color=black] (0.52773,0.28635) node {$A_2$};
	\draw [fill=black] (0,-0.634) circle (1.5pt);
	\draw[color=black] (0.0752,-0.5963) node {$A_3$};
	\draw [fill=black] (-0.5,-0.23205) circle (1.5pt);
	\draw[color=black] (-0.47268,-0.29) node {$B_1$};
	\draw [fill=black] (0.5,-0.23205) circle (1.5pt);
	\draw[color=black] (0.52773,-0.29) node {$B_2$};
	\draw [fill=black] (0,0.634) circle (1.5pt);
	\draw[color=black] (0.08,0.61) node {$B_3$};
	\draw [fill=black] (-0.23205,-0.5) circle (1.5pt);
	\draw[color=black] (-0.29,-0.44833) node {$C_1$};
	\draw [fill=black] (-0.23205,0.5) circle (1.5pt);
	\draw[color=black] (-0.29,0.55208) node {$C_2$};
	\draw [fill=black] (0.634,0) circle (1.5pt);
	\draw[color=black] (0.66216,0.06188) node {$C_3$};
	\end{scriptsize}
	\end{tikzpicture}
	\captionof{figure}{}
   \label{harom}
\end{wrapfigure}
In the case of triangles we do not need $4$ to a form a counterexample. We show a construction for $3$ triangles but omit 
the proof of correctness (Figure \ref{harom}).

Let $S$ be the biggest square that can be written into a unit 
equilateral triangle, with one line segment in common with the 
triangle. Now draw three equilateral triangle around $S$, with 
three sides of the square lying on one-one side of the 
triangles. For this construction the perimeter-to-area ratio of the 
union is $\approx 6.97 > 4 \sqrt{3}$. Actually, we get a 
counterexample even if we add the forth triangle to the figure. In 
that case the ratio is $\approx 7.06$. 

Using an optimisation algorithm, we found the following figure of a construction containing 
25 squares with ratio of about $4.28$.
\begin{center}
\definecolor{zzttqq}{rgb}{0.6,0.2,0}
\definecolor{zzttqq}{rgb}{0.6,0.2,0}
\begin{tikzpicture}[line cap=round,line join=round,>=triangle 45,x=6.0cm,y=6.0cm]
	
	\clip(-0.8,-0.40046) rectangle (0.78,1.19853);
	\fill[color=zzttqq,fill=zzttqq,fill opacity=0.1] (0.68382,0.36554) -- (-0.02329,1.07265) -- (-0.7304,0.36554) -- (-0.02329,-0.34156) -- cycle;
	\fill[color=zzttqq,fill=zzttqq,fill opacity=0.1] (0.66399,0.434) -- (-0.09031,1.09052) -- (-0.74684,0.33622) -- (0.00747,-0.3203) -- cycle;
	\fill[color=zzttqq,fill=zzttqq,fill opacity=0.1] (0.65075,0.49059) -- (-0.13941,1.10349) -- (-0.75231,0.31334) -- (0.03784,-0.29957) -- cycle;
	\fill[color=zzttqq,fill=zzttqq,fill opacity=0.1] (0.64496,0.54493) -- (-0.17533,1.11688) -- (-0.74728,0.29659) -- (0.07301,-0.27536) -- cycle;
	\fill[color=zzttqq,fill=zzttqq,fill opacity=0.1] (0.64296,0.60644) -- (-0.21093,1.12689) -- (-0.73138,0.27299) -- (0.12251,-0.24745) -- cycle;
	\fill[color=zzttqq,fill=zzttqq,fill opacity=0.1] (0.64244,0.68575) -- (-0.24857,1.13974) -- (-0.70256,0.24873) -- (0.18845,-0.20526) -- cycle;
	\fill[color=zzttqq,fill=zzttqq,fill opacity=0.1] (0.63839,0.75641) -- (-0.28359,1.14364) -- (-0.67082,0.22166) -- (0.25116,-0.16558) -- cycle;
	\fill[color=zzttqq,fill=zzttqq,fill opacity=0.1] (0.63431,0.81245) -- (-0.3092,1.14379) -- (-0.64054,0.20027) -- (0.30297,-0.13106) -- cycle;
	\fill[color=zzttqq,fill=zzttqq,fill opacity=0.1] (0.63337,0.85138) -- (-0.32692,1.13037) -- (-0.60591,0.17007) -- (0.35438,-0.10892) -- cycle;
	\fill[color=zzttqq,fill=zzttqq,fill opacity=0.1] (0.63584,0.88768) -- (-0.34007,1.10582) -- (-0.55822,0.12991) -- (0.4177,-0.08824) -- cycle;
	\fill[color=zzttqq,fill=zzttqq,fill opacity=0.1] (0.63391,0.92448) -- (-0.35486,1.07393) -- (-0.50431,0.08516) -- (0.48446,-0.06429) -- cycle;
	\fill[color=zzttqq,fill=zzttqq,fill opacity=0.1] (0.63097,0.95421) -- (-0.3653,1.04049) -- (-0.45159,0.04422) -- (0.54468,-0.04206) -- cycle;
	\fill[color=zzttqq,fill=zzttqq,fill opacity=0.1] (0.62231,0.97404) -- (-0.37719,1.00545) -- (-0.4086,0.00594) -- (0.5909,-0.02547) -- cycle;
	\fill[color=zzttqq,fill=zzttqq,fill opacity=0.1] (0.60016,0.99004) -- (-0.39935,0.95863) -- (-0.36794,-0.04087) -- (0.63157,-0.00946) -- cycle;
	\fill[color=zzttqq,fill=zzttqq,fill opacity=0.1] (0.5761,0.99975) -- (-0.41946,0.90565) -- (-0.32536,-0.08992) -- (0.67021,0.00419) -- cycle;
	\fill[color=zzttqq,fill=zzttqq,fill opacity=0.1] (0.54435,1.00049) -- (-0.44334,0.84405) -- (-0.2869,-0.14363) -- (0.70078,0.0128) -- cycle;
	\fill[color=zzttqq,fill=zzttqq,fill opacity=0.1] (0.50868,0.99851) -- (-0.46724,0.78037) -- (-0.24909,-0.19555) -- (0.72682,0.0226) -- cycle;
	\fill[color=zzttqq,fill=zzttqq,fill opacity=0.1] (0.46657,1.00044) -- (-0.49373,0.72145) -- (-0.21474,-0.23885) -- (0.74556,0.04014) -- cycle;
	\fill[color=zzttqq,fill=zzttqq,fill opacity=0.1] (0.41756,1.0059) -- (-0.52332,0.66716) -- (-0.18458,-0.27372) -- (0.7563,0.06502) -- cycle;
	\fill[color=zzttqq,fill=zzttqq,fill opacity=0.1] (0.35653,1.01048) -- (-0.56123,0.61333) -- (-0.16408,-0.30442) -- (0.75367,0.09273) -- cycle;
	\fill[color=zzttqq,fill=zzttqq,fill opacity=0.1] (0.2899,1.01276) -- (-0.6011,0.55877) -- (-0.14711,-0.33223) -- (0.74389,0.12176) -- cycle;
	\fill[color=zzttqq,fill=zzttqq,fill opacity=0.1] (0.22496,1.01637) -- (-0.63578,0.50733) -- (-0.12674,-0.35341) -- (0.734,0.15563) -- cycle;
	\fill[color=zzttqq,fill=zzttqq,fill opacity=0.1] (0.16398,1.02455) -- (-0.6631,0.46247) -- (-0.10101,-0.36461) -- (0.72607,0.19747) -- cycle;
	\fill[color=zzttqq,fill=zzttqq,fill opacity=0.1] (0.10211,1.03591) -- (-0.68805,0.423) -- (-0.07514,-0.36715) -- (0.71502,0.24576) -- cycle;
	\fill[color=zzttqq,fill=zzttqq,fill opacity=0.1] (0.04105,1.05157) -- (-0.70907,0.39025) -- (-0.04775,-0.35986) -- (0.70236,0.30145) -- cycle;
	\draw [color=zzttqq] (0.68382,0.36554)-- (-0.02329,1.07265);
	\draw [color=zzttqq] (-0.02329,1.07265)-- (-0.7304,0.36554);
	\draw [color=zzttqq] (-0.7304,0.36554)-- (-0.02329,-0.34156);
	\draw [color=zzttqq] (-0.02329,-0.34156)-- (0.68382,0.36554);
	\draw [color=zzttqq] (0.66399,0.434)-- (-0.09031,1.09052);
	\draw [color=zzttqq] (-0.09031,1.09052)-- (-0.74684,0.33622);
	\draw [color=zzttqq] (-0.74684,0.33622)-- (0.00747,-0.3203);
	\draw [color=zzttqq] (0.00747,-0.3203)-- (0.66399,0.434);
	\draw [color=zzttqq] (0.65075,0.49059)-- (-0.13941,1.10349);
	\draw [color=zzttqq] (-0.13941,1.10349)-- (-0.75231,0.31334);
	\draw [color=zzttqq] (-0.75231,0.31334)-- (0.03784,-0.29957);
	\draw [color=zzttqq] (0.03784,-0.29957)-- (0.65075,0.49059);
	\draw [color=zzttqq] (0.64496,0.54493)-- (-0.17533,1.11688);
	\draw [color=zzttqq] (-0.17533,1.11688)-- (-0.74728,0.29659);
	\draw [color=zzttqq] (-0.74728,0.29659)-- (0.07301,-0.27536);
	\draw [color=zzttqq] (0.07301,-0.27536)-- (0.64496,0.54493);
	\draw [color=zzttqq] (0.64296,0.60644)-- (-0.21093,1.12689);
	\draw [color=zzttqq] (-0.21093,1.12689)-- (-0.73138,0.27299);
	\draw [color=zzttqq] (-0.73138,0.27299)-- (0.12251,-0.24745);
	\draw [color=zzttqq] (0.12251,-0.24745)-- (0.64296,0.60644);
	\draw [color=zzttqq] (0.64244,0.68575)-- (-0.24857,1.13974);
	\draw [color=zzttqq] (-0.24857,1.13974)-- (-0.70256,0.24873);
	\draw [color=zzttqq] (-0.70256,0.24873)-- (0.18845,-0.20526);
	\draw [color=zzttqq] (0.18845,-0.20526)-- (0.64244,0.68575);
	\draw [color=zzttqq] (0.63839,0.75641)-- (-0.28359,1.14364);
	\draw [color=zzttqq] (-0.28359,1.14364)-- (-0.67082,0.22166);
	\draw [color=zzttqq] (-0.67082,0.22166)-- (0.25116,-0.16558);
	\draw [color=zzttqq] (0.25116,-0.16558)-- (0.63839,0.75641);
	\draw [color=zzttqq] (0.63431,0.81245)-- (-0.3092,1.14379);
	\draw [color=zzttqq] (-0.3092,1.14379)-- (-0.64054,0.20027);
	\draw [color=zzttqq] (-0.64054,0.20027)-- (0.30297,-0.13106);
	\draw [color=zzttqq] (0.30297,-0.13106)-- (0.63431,0.81245);
	\draw [color=zzttqq] (0.63337,0.85138)-- (-0.32692,1.13037);
	\draw [color=zzttqq] (-0.32692,1.13037)-- (-0.60591,0.17007);
	\draw [color=zzttqq] (-0.60591,0.17007)-- (0.35438,-0.10892);
	\draw [color=zzttqq] (0.35438,-0.10892)-- (0.63337,0.85138);
	\draw [color=zzttqq] (0.63584,0.88768)-- (-0.34007,1.10582);
	\draw [color=zzttqq] (-0.34007,1.10582)-- (-0.55822,0.12991);
	\draw [color=zzttqq] (-0.55822,0.12991)-- (0.4177,-0.08824);
	\draw [color=zzttqq] (0.4177,-0.08824)-- (0.63584,0.88768);
	\draw [color=zzttqq] (0.63391,0.92448)-- (-0.35486,1.07393);
	\draw [color=zzttqq] (-0.35486,1.07393)-- (-0.50431,0.08516);
	\draw [color=zzttqq] (-0.50431,0.08516)-- (0.48446,-0.06429);
	\draw [color=zzttqq] (0.48446,-0.06429)-- (0.63391,0.92448);
	\draw [color=zzttqq] (0.63097,0.95421)-- (-0.3653,1.04049);
	\draw [color=zzttqq] (-0.3653,1.04049)-- (-0.45159,0.04422);
	\draw [color=zzttqq] (-0.45159,0.04422)-- (0.54468,-0.04206);
	\draw [color=zzttqq] (0.54468,-0.04206)-- (0.63097,0.95421);
	\draw [color=zzttqq] (0.62231,0.97404)-- (-0.37719,1.00545);
	\draw [color=zzttqq] (-0.37719,1.00545)-- (-0.4086,0.00594);
	\draw [color=zzttqq] (-0.4086,0.00594)-- (0.5909,-0.02547);
	\draw [color=zzttqq] (0.5909,-0.02547)-- (0.62231,0.97404);
	\draw [color=zzttqq] (0.60016,0.99004)-- (-0.39935,0.95863);
	\draw [color=zzttqq] (-0.39935,0.95863)-- (-0.36794,-0.04087);
	\draw [color=zzttqq] (-0.36794,-0.04087)-- (0.63157,-0.00946);
	\draw [color=zzttqq] (0.63157,-0.00946)-- (0.60016,0.99004);
	\draw [color=zzttqq] (0.5761,0.99975)-- (-0.41946,0.90565);
	\draw [color=zzttqq] (-0.41946,0.90565)-- (-0.32536,-0.08992);
	\draw [color=zzttqq] (-0.32536,-0.08992)-- (0.67021,0.00419);
	\draw [color=zzttqq] (0.67021,0.00419)-- (0.5761,0.99975);
	\draw [color=zzttqq] (0.54435,1.00049)-- (-0.44334,0.84405);
	\draw [color=zzttqq] (-0.44334,0.84405)-- (-0.2869,-0.14363);
	\draw [color=zzttqq] (-0.2869,-0.14363)-- (0.70078,0.0128);
	\draw [color=zzttqq] (0.70078,0.0128)-- (0.54435,1.00049);
	\draw [color=zzttqq] (0.50868,0.99851)-- (-0.46724,0.78037);
	\draw [color=zzttqq] (-0.46724,0.78037)-- (-0.24909,-0.19555);
	\draw [color=zzttqq] (-0.24909,-0.19555)-- (0.72682,0.0226);
	\draw [color=zzttqq] (0.72682,0.0226)-- (0.50868,0.99851);
	\draw [color=zzttqq] (0.46657,1.00044)-- (-0.49373,0.72145);
	\draw [color=zzttqq] (-0.49373,0.72145)-- (-0.21474,-0.23885);
	\draw [color=zzttqq] (-0.21474,-0.23885)-- (0.74556,0.04014);
	\draw [color=zzttqq] (0.74556,0.04014)-- (0.46657,1.00044);
	\draw [color=zzttqq] (0.41756,1.0059)-- (-0.52332,0.66716);
	\draw [color=zzttqq] (-0.52332,0.66716)-- (-0.18458,-0.27372);
	\draw [color=zzttqq] (-0.18458,-0.27372)-- (0.7563,0.06502);
	\draw [color=zzttqq] (0.7563,0.06502)-- (0.41756,1.0059);
	\draw [color=zzttqq] (0.35653,1.01048)-- (-0.56123,0.61333);
	\draw [color=zzttqq] (-0.56123,0.61333)-- (-0.16408,-0.30442);
	\draw [color=zzttqq] (-0.16408,-0.30442)-- (0.75367,0.09273);
	\draw [color=zzttqq] (0.75367,0.09273)-- (0.35653,1.01048);
	\draw [color=zzttqq] (0.2899,1.01276)-- (-0.6011,0.55877);
	\draw [color=zzttqq] (-0.6011,0.55877)-- (-0.14711,-0.33223);
	\draw [color=zzttqq] (-0.14711,-0.33223)-- (0.74389,0.12176);
	\draw [color=zzttqq] (0.74389,0.12176)-- (0.2899,1.01276);
	\draw [color=zzttqq] (0.22496,1.01637)-- (-0.63578,0.50733);
	\draw [color=zzttqq] (-0.63578,0.50733)-- (-0.12674,-0.35341);
	\draw [color=zzttqq] (-0.12674,-0.35341)-- (0.734,0.15563);
	\draw [color=zzttqq] (0.734,0.15563)-- (0.22496,1.01637);
	\draw [color=zzttqq] (0.16398,1.02455)-- (-0.6631,0.46247);
	\draw [color=zzttqq] (-0.6631,0.46247)-- (-0.10101,-0.36461);
	\draw [color=zzttqq] (-0.10101,-0.36461)-- (0.72607,0.19747);
	\draw [color=zzttqq] (0.72607,0.19747)-- (0.16398,1.02455);
	\draw [color=zzttqq] (0.10211,1.03591)-- (-0.68805,0.423);
	\draw [color=zzttqq] (-0.68805,0.423)-- (-0.07514,-0.36715);
	\draw [color=zzttqq] (-0.07514,-0.36715)-- (0.71502,0.24576);
	\draw [color=zzttqq] (0.71502,0.24576)-- (0.10211,1.03591);
	\draw [color=zzttqq] (0.04105,1.05157)-- (-0.70907,0.39025);
	\draw [color=zzttqq] (-0.70907,0.39025)-- (-0.04775,-0.35986);
	\draw [color=zzttqq] (-0.04775,-0.35986)-- (0.70236,0.30145);
	\draw [color=zzttqq] (0.70236,0.30145)-- (0.04105,1.05157);
\end{tikzpicture}

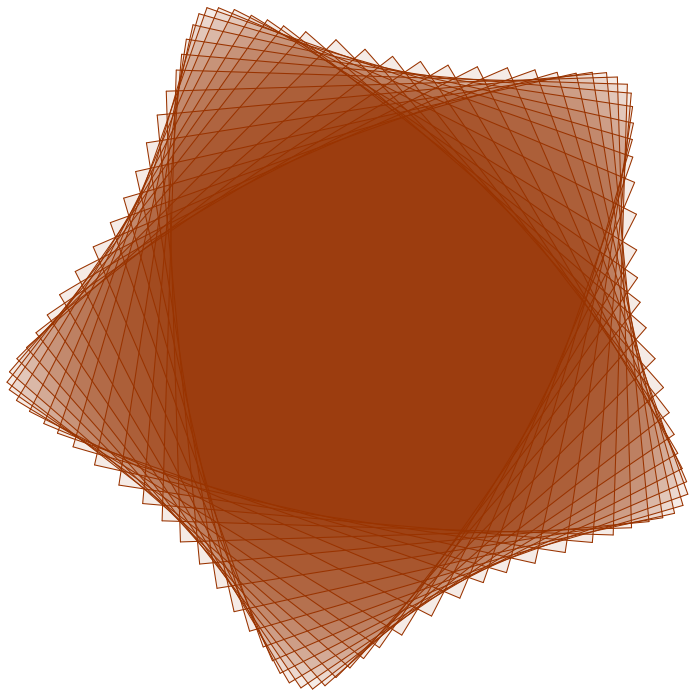
\captionof{figure}{$25$ squares with ratio $\approx 4.28$}
   \label{25}
\end{center}

\section{Open problems}

We believe that the constructions containing five squares and four 
equilateral triangles can be generalised to give a counterexample 
of $n + 1$ regular polygons with $n$ vertices and unit side length.

\begin{conjecture}
  The perimeter-to-area ratio of a shifted $(n + 1,n)$-setup is 
  greater than the ratio in the case of a single regular $n$-gon with unit side length.
\end{conjecture}

According to the results obtained by computer, we think the following is true:
\begin{conjecture}
  Let $k$ and $n$ be coprime natural numbers. Then the perimeter-to-area ratio of a shifted $(k ,n)$-setup is 
  greater than the ratio in the case of a single regular $n$-gon with unit side length if and only if $k > 1$ and $k \equiv 1 \; (\text{mod }n)$.
\end{conjecture}

Since Keleti's boundedness result works in $\R^d$ as well it is natural to ask the following:

\begin{question}
  Do the analogous constructions work in higher dimensions for regular polyhedrons?
\end{question}

As we saw it in Section \ref{s:other} there is an example of four squares with perimeter-to-area 
ratio greater than 4, but we could not find any using only three. 
\begin{question}
  What is the minimum number of squares that form such an example? Or in general, what is the minimum
  number of regular $n$-gons for which the perimeter-to-area ratio of the union is greater than the ratio in case of a single regular $n$-gon? Is it equal to $n$?
\end{question}

Since we have found the first examples with a probabilistic computer algorithm, it would be interesting to know that whether
the 'majority' of the setups close to the basic setup give an example in some sense. For $k$-many regular $n$-gons let us denote the centres by $C_1, C_2, \dots , C_k$ and the rotations of the polygons by $r_1,\dots,r_k$ respectively. Let $f_{k,n}(C_1,\dots,C_k,r_1,\dots,r_k)$ be the perimeter-to-area ratio of such a setup and $p_0=(C^0_1,\dots,C^0_k,r^0_1,\dots,r^0_k)$ the centres and rotations of the basic $(k,n)$-setup (of course, $C^0_i=(0,0)$). 

\begin{question}
  What are the derivatives of the function $f_{k,n}$ at the point $p_0$? Is this point a local minimum of $f_{k,n}$? Or does there exist a neighbourhood of $p_0$ where almost all points form an example with greater perimeter-to-area ratio than the basic setup?
\end{question}

We also could not go close to the current best upper bound of about 5.6 
proved for the ratio by Gyenes. The best example we could find with the help 
of a computer has ratio about $4.34$ and contains $100$ squares. Since with increased number of squares we could increase significantly the ratio it may happen that the lowest upper bound for the ratio cannot be realised by an a single example. 
\begin{question}
  What is the optimal upper bound for the ratio?  Does there exist a construction which maximises the ratio? 
\end{question}
 Even if there is no such a construction, it is possible that there exists a construction which maximises
the ratio for a fixed number of polygons.
\begin{question}
 What is the optimal upper bound for the ratio if we can use $k$ many polygons?  Does there exist a construction which maximises the ratio in this case? 
\end{question}
 Finally, except for the result of Gyenes concerning the union of discs, there is no example of a compact convex set, for which the analogous question to Keleti's conjecture would be true. 
 
 \begin{question}
  Does there exist a compact convex set $C$ apart from a disc, so that $p(C)/a(C)$ maximises the perimeter-to-area ratio among the finite unions of sets congruent to $C$?
 \end{question}

\end{document}